\documentclass[12pt]{amsart}
\usepackage[hmargin=3cm,vmargin=3.5cm]{geometry}

\usepackage[dvips]{epsfig}
\usepackage{pinlabel}
\usepackage{xy}
\usepackage{bbm}
\usepackage{amssymb}

\newcommand{\ot}{\otimes}
\newcommand{\ii}{\underline{\textbf{\textit{i}}}}
\newcommand{\jj}{\underline{\textbf{\textit{j}}}}
\newcommand{\kk}{\underline{\textbf{\textit{k}}}}
\newcommand{\mc}[1]{\mathcal{#1}}
\newcommand{\mC}{\mathcal{C}}

\newcommand{\mH}{\mathcal{H}}
\newcommand{\mTL}{\mathcal{TL}}
\newcommand{\mHC}{\mathcal{HC}}
\newcommand{\mTLC}{\mathcal{TLC}}
\newcommand{\mV}{\mathcal{V}}
\newcommand{\mL}{\mathcal{L}}

\newcommand{\qtwo}{\left[2\right]}
\newcommand{\Foam}{\rm{Foam}}

\newcommand{\Hom}{{\rm Hom}}
\newcommand{\Mor}{{\rm Mor}}
\newcommand{\HOM}{{\rm HOM}}

\renewcommand{\to}{\rightarrow}

\newcommand{\id}{{\rm id}}
\newcommand{\define}{\stackrel{\mbox{\scriptsize{def}}}{=}}
\newcommand{\End}{{\rm End}}
\newcommand{\gdim}{{\rm gdim}}     

\newcommand{\Rbim}{R{\rm -bim}}

\newcommand{\Ztt}{\ensuremath{\mathbbm{Z}\left[t,t^{-1}\right]}}
\newcommand{\Zttt}{\ensuremath{\mathbbm{Z}\left[\left[t, t^{-1}\right]\right]}}

\def\C{{\mathbbm C}}
\def\N{{\mathbbm N}}
\def\R{{\mathbbm R}}
\def\Z{{\mathbbm Z}}

\def\1{\mathbbm{1}}

\let\hat=\widehat
\let\tilde=\widetilde
\let\phi=\varphi
\let\epsilon=\varepsilon

\newcommand{\ig}[2]{\vcenter{\xy (0,0)*{\includegraphics[scale=#1]{fig/#2}} \endxy}}
\newcommand{\igv}[2]{\vcenter{\xy (0,0)*{\reflectbox{\includegraphics[scale=#1, angle=180]{fig/#2}}} \endxy}}
\newcommand{\igh}[2]{\vcenter{\xy (0,0)*{\reflectbox{\includegraphics[scale=#1]{fig/#2}}} \endxy}}

\newcommand{\igc}[2]{\begin{center} \includegraphics[scale=#1]{fig/#2} \end{center}}

\newcommand{\plabel}[3]{{
\labellist
\small\hair 2pt
\pinlabel #1 at #2 #3
\endlabellist
}}

\newcommand{\plabeltwo}[6]{{
\labellist
\small\hair 2pt
\pinlabel #1 at #2 #3
\pinlabel #4 at #5 #6
\endlabellist
}}

\newcommand{\plabelthree}[9]{{
\labellist
\small\hair 2pt
\pinlabel #1 at #2 #3
\pinlabel #4 at #5 #6
\pinlabel #7 at #8 #9
\endlabellist
}}

\newtheorem{prop}{Proposition}[section]
\newtheorem{lemma}[prop]{Lemma}
\newtheorem{thm}[prop]{Theorem}
\newtheorem{cor}[prop]{Corollary}

\newtheorem{claim}[prop]{Claim}

\theoremstyle{definition}
\newtheorem{defn}[prop]{Definition}
\newtheorem{example}[prop]{Example}
\newtheorem{notation}[prop]{Notation}

\theoremstyle{remark}
\newtheorem{remark}[prop]{Remark}

\numberwithin{equation}{section}

\title{A Diagrammatic Temperley-Lieb Categorification}
\author{Ben Elias}
 
\begin{document} 

\baselineskip 14pt
\maketitle
 
\setcounter{tocdepth}{1}
\tableofcontents

\begin{abstract} The monoidal category of Soergel bimodules categorifies the Hecke algebra of a finite Weyl group. In the case of the symmetric group, morphisms in this category can be
drawn as graphs in the plane. We define a quotient category, also given in terms of planar graphs, which categorifies the Temperley-Lieb algebra. Certain ideals appearing in this
quotient are related both to the 1-skeleton of the Coxeter complex and to the topology of 2D cobordisms. We demonstrate how further subquotients of this category will categorify the
irreducible modules of the Temperley-Lieb algebra. \end{abstract}

\section{Introduction} 

A goal of the categorification theorist is to replace interesting endomorphisms of a vector space with interesting endofunctors of a category. The question is: what makes these functors
interesting? In the pivotal paper of Chuang and Rouquier \cite{CR}, a fresh paradigm emerged. They noticed that by specifying structure on the natural transformations (morphisms) between
these functors one obtains more useful categorifications (in this case, the added utility is a certain derived equivalence). The categorification of quantum groups by Rouquier
\cite{Rou3}, Lauda \cite{Lau}, and Khovanov and Lauda \cite{KL3} has shown that categorifying an algebra $A$ itself (with a category $\mc{A}$) will specify what this additional structure
should be for a categorification of any representation of that algebra: a functor from $\mc{A}$ to an endofunctor category. That their categorifications $\mc{A}$ provide the ``correct"
extra structure is confirmed by the fact that existing geometric categorifications conform to it (see \cite{VV}) and that irreducible representations of $A$ can be categorified in
this framework (see \cite{LV,HS}). The salient feature of these categorifications is that, instead of being defined abstractly, the morphisms are presented by generators and relations,
making it straightforward to define functors out of $\mc{A}$.

In the case of the Hecke algebra $\mH$, categorifications have existed for some time, in the guise of category $\mc{O}$ or perverse sheaves on the flag variety. In \cite{Soe1} Soergel
rephrased these categorifications in a more combinatorial way, constructing an additive categorification of $\mH$ by a certain full monoidal subcategory $\mHC$ of graded $R$-bimodules,
where $R$ is a polynomial ring. Objects in this full subcategory are called \emph{Soergel bimodules}. There are deep connections between Soergel bimodules, representation theory, and
geometry, and we refer the reader to \cite{Soe1,Soe2,Soe3,Soe4} for more details. Categorifications using category $\mc{O}$ and variants thereof are common in the literature, and often Soergel bimodules are used to aid calculations (see, for instance, \cite{KMS,MS,Str}).

In \cite{EKh}, the author and M. Khovanov provide (in type $A$) a presentation of $\mHC$ by generators and relations, where morphisms can be viewed diagrammatically as decorated graphs in
a plane. To be more precise, the diagrammatics are for a smaller category $\mc{HC}_1$, the (ungraded) category of \emph{Bott-Samelson bimodules}, described in Section \ref{subsec-hecke}.
Soergel bimodules are obtained from $\mHC_1$ by taking the graded Karoubi envelope. This is in exact analogy with the procedures of Khovanov and Lauda in \cite{KL3} and related papers.

The Temperley-Lieb algebra $\mTL$ is a well-known quotient of $\mH$, and it can be categorified by a quotient $\mTLC$ of $\mHC$, as this paper endeavors to show. Thus we have a naturally
arising categorification by generators and relations, and we expect it to be a useful one. Objects in $\mTLC$ can no longer be viewed as $R$-bimodules (though their Hom spaces will be
$R$-bimodules), so that diagrammatics provide the simplest way to define the category.

The most complicated generator of $\mHC$ is killed in the quotient to $\mTLC$, making $\mTLC$ easy to describe diagrammatically in its own right. Take a category where objects are
sequences of indices between $1$ and $n$ (denoted $\ii$). Morphisms will be given by (linear combinations of) collections of graphs $\Gamma_i$ embedded in $\R \times [0,1]$, one for each
index $i \in \{1,\ldots,n\}$, such that the graphs have only trivalent or univalent vertices, and such that $\Gamma_i$ and $\Gamma_{i+1}$ are disjoint. Each graph will have a degree,
making Hom spaces into a graded vector space. The intersection of the graphs with $\R \times \{0\}$ and $\R \times \{1\}$ determine the source and target objects respectively. Finally,
some local graphical relations are imposed on these morphisms. This defines $\mTLC_1$, and we take the graded Karoubi envelope to obtain $\mTLC$.

The proof that $\mTLC$ categorifies $\mTL$ uses a method similar to that in \cite{EKh}. We show first that
$\mTLC_1$ is a \emph{potential categorification} of $\mTL$, in the sense described in section
\ref{subsec-temperley}. Categorifications and potential categorifications define a pairing on $\mTL$ given by
$([M],[N])=\gdim \Hom_{\mTLC_1}(M,N)$, the graded dimension which takes values in $\Zttt$. Equivalently, it
defines a trace on $\mTL$ via $\epsilon([M])=\gdim \Hom(\1,M)$ where $\1$ is the monoidal identity (see Section
\ref{subsec-hecke}). The difficult part is to prove the following lemma.

\begin{lemma} \label{mainlemma} The trace induced on $\mTL$ from $\mTLC_1$ is the map $\epsilon_{\rm{cat}}$ defined in Section \ref{subsec-temperley}. \end{lemma}

Given this lemma, it is surprisingly easy (see Section \ref{subsec-usingpairing}) to show the main theorem.

\begin{thm} \label{maintheorem} Let $\mTLC_2$ be the graded additive closure of $\mTLC_1$, and let $\mTLC$ be the graded Karoubi envelope of $\mTLC_1$. Then $\mTLC_2$ is Krull-Schmidt
and idempotent-closed, so $\mTLC_2 \cong \mTLC$, and $\mTLC$ categorifies $\mTL$. \end{thm}

To prove the lemma, we note that there is a convenient set of elements in $\mTL$, the \emph{non-repeating} monomials, whose values determine any pairing; hence, there is a convenient set
of objects whose Hom spaces will determine all Hom spaces. If $\ii$ is a non-repeating sequence, the Hom space we must calculate is (up to shift) a quotient
of $R$ by a two-sided ideal $I_{\ii}$. We use graphical methods to determine these rings explicitly, giving generators for the ideals in $R$ which define them. As an interesting side
note, these ideals also occur elsewhere in nature.

\begin{prop} Let $V$ be the reflection representation of $S_{n+1}$, and identify $R$ with its coordinate ring. Let $Z$ be the union of all the lines in $V$ which are intersections of
reflection-fixed hyperplanes, and let $I \subset R$ be the ideal which gives the reduced scheme structure on $Z$. Then Hom spaces in $\mTLC$ are $R/I$-bimodules, and the ideals $I_{\ii}$
cut out subvarieties of $Z$ given by lines with certain transverseness properties (see Section \ref{subsec-weyllines} for details). \end{prop}

Also in Section \ref{subsec-weyllines} we give a topological interpretation of the ideals $I_{\ii}$, using a functor defined by Vaz \cite{Vaz}.

Now, let $\mTL(J_i)$ be the parabolic subalgebra of $\mTL$ given by ignoring the index $i$, and let $V^i$ be the induced (right) representation from the sign representation of
$\mTL(J_i)$. Such an induced representation is useful because it is a quotient of $\mTL$, and also contains an irreducible module $L^i$ of $\mTL$ as a submodule. All irreducibles can be
constructed this way.

We provide a diagrammatic categorification of $V^i$ as a quotient $\mV^i$ of $\mTLC$, and a categorification of $L^i$ as a full subcategory $\mL^i$ of $\mV^i$, in a fashion analogous to
quantum group categorifications. Having found a diagrammatic categorification $\mC$ of the positive half $U^+$ of the quantum group, Khovanov and Lauda in \cite{KL1} conjectured that
highest weight modules (naturally quotients of $U^+$) could be categorified by quotients of $\mC$ by the appearance of certain pictures on the \emph{left}. This approach was proven
correct by Lauda and Vazirani \cite{LV} (for the $U^+$-module structure), and then used by Webster to categorify tensor products \cite{Web}. Similarly, to obtain $\mV^i$ we mod out
$\mTLC$ by diagrams where any index except $i$ appears on the left. The proof that this works is similar in style to the proof of Theorem \ref{maintheorem}: one calculates the dimension of all Hom spaces by calculating enough Hom spaces to specify a unique pairing on $V^i$, and then uses simple arguments to identify the Grothendieck group.

\begin{thm} The category $\mV^i$ is idempotent-closed and Krull-Schmidt. Its Grothendieck group is isomorphic to $V^i$. Letting $\mL^i$ be the full subcategory generated by
indecomposables which decategorify to elements of $L^i$, we have that $\mL^i$ is idempotent-closed and Krull-Schmidt, with Grothendieck group isomorphic to $L^i$. \end{thm}

A future paper will categorify all representations induced from the sign and trivial representations of parabolic subalgebras of $\mH$ and
$\mTL$. Induced representations were categorified more generally in \cite{MS} in the context of category $\mc{O}$, although not
diagrammatically. We believe that our categorification should describe what happens in \cite{MS} after applying Soergel's functor.

Soergel bimodules are intrinsically linked with braids, as was shown by Rouquier in \cite{Rou1, Rou2}, who used them to construct braid group
actions (these braid group actions also appear in the category $\mc{O}$ context, see \cite{AS}). As such, morphisms between Soergel bimodules
should correspond roughly to movies, and the graphs appearing in the diagrammatic presentation of the category $\mc{HC}$ should be
(heuristically) viewed as 2-dimensional holograms of braid cobordisms. This is studied in \cite{EKr}. The Temperley-Lieb quotient is associated
with the representation theory of $U_q(\mathfrak{sl}_2)$, for which braids all degenerate into 1-manifolds, and braid cobordisms degenerate into
surfaces with disorientations. There is a functor $\mc{F}$ from $\mc{TLC}$ to the category of disorientations constructed by Vaz \cite{Vaz}. The
functor $\mc{F}$ is faithful (though certainly not full) as we remark in Section \ref{subsec-weyllines}. This in turn yields a topological
motivation of the variety $Z$ and its subvarieties $Z^\prime$. Because $\mc{F}$ is not full, there might be actions of $\mTLC$ that do not
extend to actions of disoriented cobordisms. Cobordisms have long been a reasonable candidate for morphisms in Temperley-Lieb categorifications,
although we hope $\mTLC$ will provide a useful substitute, with more explicit and computable Hom spaces.

Categorification and the Temperley-Lieb algebra have a long history. Khovanov in \cite{Kho} constructed a categorification of $\mTL$ using a
TQFT, which was slightly generalized by Bar-Natan in \cite{BN}. This was then used to categorify the Jones polynomial. Bernstein, Frenkel and
Khovanov in \cite{BFK} provide a categorical action of the Temperley-Lieb algebra by Zuckerman and projective functors on category $\mc{O}$.
Stroppel \cite{Str} showed that this categorical action extends to the full tangle algebroid, and also investigated the natural transformations
between projective functors. Recent work of Brundan and Stroppel \cite{BS} connects these Temperley-Lieb categorifications to
Khovanov-Lauda-Rouquier algebras, among other things. We hope that our diagrammatics will help to understand the morphisms in these
categorifications.

The organization of this paper is as follows. Chapter \ref{sec-preliminaries} will provide a quick overview of the Hecke and Temperley-Lieb
algebras, and the diagrammatic definition of the category $\mc{HC}$. Chapter \ref{sec-tlc} begins by defining the quotient category
diagrammatically in its own right (which makes a thorough understanding of the diagrammatic calculus for $\mc{HC}$ unnecessary). Section
\ref{subsec-usingpairing} proves Theorem \ref{maintheorem}, modulo Lemma \ref{mainlemma} which requires all the work. The remaining sections of
that chapter do all the work, and starting with Section \ref{subsec-gradeddim} one will not miss any important ideas if one skips the proofs.
Chapter \ref{sec-repns} begins with a discussion of cell modules for $\mc{TL}$ and certain other modules, and then goes on to categorify these
modules, requiring only very simple diagrammatic arguments.

This paper is reasonably self-contained. We do not require familiarity with \cite{EKh}, and do not use any
results other than Corollary \ref{freemodules} below. We do quote some results for motivational reasons, but
the difficult graphical arguments of that paper can often be drastically simplified for the Temperley-Lieb
setting, so that we provide easier proofs for the results we need. Familiarity with diagrammatics for monoidal
categories with adjunction would be useful, and \cite{Lau} provides a good introduction. More details on
preliminary topics can be found in \cite{EKh}.  

\vspace{0.06in}

{\bf Acknowledgments.} 

The author was supported by NSF grants DMS-524460 and DMS-524124. The author would like to thank Mikhail
Khovanov, Catharina Stroppel, and the referees for many thoughtful comments.

\section{Preliminaries}
\label{sec-preliminaries}

\begin{notation} Fix $n \in \N$, and let $I={1,\ldots,n}$ index the vertices of the Dynkin diagram $A_n$. We use the word \emph{index} for an element of $I$, and the letters
$i,j$ always represent indices. Indices $i \ne j$ are \emph{adjacent} if $|i-j|=1$, and \emph{distant} if $|i-j| \ge 2$, and questions of \emph{adjacency} always
refer to the Dynkin diagram, not the position of indices in a word or picture. \end{notation}

\begin{notation} Let $W=S_{n+1}$ with simple reflections $s_i = (i, i+1)$. Let $\Bbbk$ be a field of characteristic not dividing $2(n+1)$; all vector spaces will be over this field. Let
$R=\Bbbk[x_1,\ldots, x_{n+1}]/e_1$, where $e_1=x_1 + x_2 + \ldots + x_{n+1}$; it is a graded ring, with $\deg(x_i)=2$. We will abuse notation and refer to elements of $\Bbbk[x_1, \ldots,
x_{n+1}]$ and their images in $R$ in the same way, and will refer to both as \emph{polynomials}. Note that $R=\Bbbk[f_1,\ldots, f_{n}]$ where $f_i=x_i - x_{i+1}$, since
$x_1=\frac{nf_1+(n-1)f_2 + \ldots + f_n}{n+1}$ modulo $e_1$. The ring $R$ arises as the coordinate ring of $V$, the reflection representation of $W$ (the span of the root system), and
$f_i$ are the simple coroots.

There is an obvious action of $S_{n+1}$ on $R$, which permutes the generators $x_i$. For each index we have a \emph{Demazure operator} $\partial_i$, a map of degree $-2$ from
$R$ to the invariant subring $R^{s_i}$, which is $R^{s_i}$-linear and sends $R^{s_i}$ to $0$. Explicitly,  $\partial_i(f)=\frac{f - s_i(f)}{x_i-x_{i+1}}$.
\end{notation}

\begin{notation} Let $\overline{(\cdot)}$ be the $\Z$-linear involution of $\Ztt$ switching $t$ and $t^{-1}$. Given a $\Z$-linear map $\beta$ of $\Ztt$-modules, we call it
\emph{antilinear} if it is $\Ztt$-linear after twisting by $\overline{(\cdot)}$, or in other words if $\beta(tm)=t^{-1}\beta(m)$. 	We write $\qtwo \define t + t^{-1}$.

Let $A$ be a $\Ztt$-algebra. In this paper we always use the word \emph{trace} to designate a $\Ztt$-linear map $\epsilon \colon A \to \Zttt$ satisfying $\epsilon(xy)=\epsilon(yx)$. We
use the word \emph{pairing} or \emph{semilinear pairing} to denote a $\Z$-linear map $A \times A \to \Zttt$ which is $\Ztt$-linear in the second factor and $\Ztt$-antilinear in the
first factor. \end{notation}

%
\subsection{The Hecke algebra and the Soergel categorification}
\label{subsec-hecke}
%

We state here without proof a number of basic facts about the Hecke algebra, its traces, and Soergel's categorification. For more background, see Soergel's original definition of his
categorification \cite{Soe1}, or an easier version \cite{Soe4}. A similar overview with more discussion can be found in \cite{EKh}. A more in-depth introduction, connecting Soergel
bimodules to other parts of representation theory, can be found in \cite{MS}.

\begin{defn} Denote by $\mH$ the \emph{Hecke algebra} for $S_{n+1}$. It is a $\Ztt$-algebra, specified here by its \emph{Kazhdan-Lusztig presentation}: it has generators $b_i, i
\in I$ and relations

\begin{eqnarray}
b_i^2 & = & (t + t^{-1}) b_i  \label{eqn-bisq}\\
b_ib_j & = & b_jb_i \ \textrm{ for distant } i, j \label{eqn-bibj}\\
b_ib_jb_i + b_j & = & b_jb_ib_j + b_i \textrm{ for adjacent } i, j \label{eqn-bibpbi}.
\end{eqnarray}
\end{defn}

\begin{defn} Given two objects in a graded $\Bbbk$-linear (possibly additive) category $\mC$, where $\{1\}$ denotes the grading shift, the \emph{graded hom space} between them is the
graded vector space $\HOM(M,N)=\oplus_{n \in \Z} \Hom_{\mC}(M,N\{n\})$. Given a class of objects $\{M_{\alpha}\}$ in $\mC$, we can define a category with morphisms enriched in graded
vector spaces, whose objects are $\{M_{\alpha}\}$ and whose morphisms are $\HOM(M_\alpha,M_\beta)$. Let us call this an \emph{enriched full subcategory}, which we often shorten to the
adjective \emph{enriched}. While the enriched subcategory is neither additive nor graded, it has enough information to recover the hom spaces between grading shifts and direct sums of
objects $M_\alpha$ in $\mC$.

Let $\Rbim$ denote the category of finitely-generated graded (resp. ungraded) $R$-bimodules. Then $\HOM$ spaces in $\Rbim$ will be graded $R$-bimodules. For $i \in I$, let
$B_i \in \Rbim$ be defined by $B_i=R \ot_{R^{s_i}} R \{-1\}$, where $R^{s_i}$ is the invariant subring. A \emph{Bott-Samelson bimodule} is a tensor product $B_{i_1} \ot B_{i_2} \ot \cdots
\ot B_{i_d}$ in $\Rbim$, where here and henceforth $\ot$ denotes the tensor product over $R$. Let $\mHC_1$ be the enriched full subcategory generated by the Bott-Samelson bimodules; it is
a monoidal category, but is neither additive nor graded. Let $\mHC_2$ denote the full subcategory of $\Rbim$ given by all (finite) direct sums of grading shifts of Bott-Samelson
bimodules; it is monoidal, additive, and graded. Finally, let $\mHC$ denote the category of \emph{Soergel bimodules} or \emph{special bimodules}, the full subcategory of $\Rbim$ given by
all (finite) direct sums of grading shifts of \emph{summands} of Bott-Samelson bimodules; it is monoidal, additive, graded, and idempotent-closed. \end{defn}

One can observe that all bimodules in $\mHC$ are free and finitely generated when viewed as either left $R$-modules or right $R$-modules, and therefore the same is true of any $\HOM$
space. The following proposition parallels the Kazhdan-Lusztig presentation for $\mH$.

\begin{prop}
The category $\mHC_2$ is generated (as an additive, monoidal category) by objects $B_i, i \in I$ which satisfy

\begin{eqnarray}
B_i \otimes B_i & \cong & B_i\{1\} \oplus B_i\{-1\} \label{dc-ii} \\
B_i \otimes B_j & \cong & B_j \otimes B_i \label{dc-ij} \textrm{ for distant } i, j \\
B_i \otimes B_j \otimes B_i \oplus B_j & \cong & B_j \otimes B_i \otimes B_j \oplus B_i \label{dc-ipi} \textrm{ for adjacent } i, j. \end{eqnarray}
\end{prop}

From this we might expect the next result. 

\begin{prop} The Grothendieck ring $[\mHC_2]$ of $\mHC_2$ is isomorphic to $\mH$, with $[B_i]$ being sent to $b_i$, and $[R\{1\}]$ being sent to $t$. The Grothendieck ring $[\mHC]$ of
$\mHC$ is isomorphic to $\mH$ as well. \end{prop}

\begin{remark} \label{rmk-grothofsoergel} The proof of this statement is not immediately obvious. There is clearly a surjective morphism from $\mH$ to $[\mHC_2]$. When one takes the
idempotent closure of a category, one adds new indecomposables and can potentially enlarge the Grothendieck group. Soergel showed, via a support filtration, that all the new
indecomposables in $\mHC$ have symbols in $[\mHC]$ which can be reached from certain symbols in $[\mHC_2]$ by a unitriangular matrix (see \cite{Soe4}). Therefore, the Grothendieck
rings of $\mHC$ and $\mHC_2$ are equal. Since $\mHC$ is idempotent-closed and is embedded in $\Rbim$, it has the Krull-Schmidt property and the Grothendieck group behaves as one would
expect: it has a basis given by indecomposables. By classifying indecomposables and using the unitriangular matrix, Soergel showed that the map from $\mH$ to $[\mHC_2]$ is actually an
isomorphism.

It is important to note that one does not know what the image of the indecomposables of $\mHC$ are in $\mH$. The Soergel conjecture, still unproven in generality, proposes that the
indecomposables of $\mHC$ descend to the Kazhdan-Lusztig basis of $\mH$ (see \cite{Soe4}).
\end{remark}

\begin{notation} We write the monomial $b_{i_1}b_{i_2}\cdots b_{i_d} \in \mH$ as $b_{\ii}$ where $\ii=i_1\ldots i_d$ is a finite sequence of indices; by abuse of notation, we sometimes
refer to this monomial simply as $\ii$. If $\ii$ is as above, we say the monomial has \emph{length} $d=d(\ii)$. We call a monomial \emph{non-repeating} if $i_k \ne i_l$ for $k \ne l$,
and \emph{increasing} if $i_1 < i_2 < \ldots$. The empty set is a sequence of length 0, and $b_{\emptyset}=1$. Similarly, in $\mHC_1$, write $B_{i_1} \otimes \ldots \otimes B_{i_d}$ as
$B_{\ii}$. Note that $B_{\emptyset}=R$, the monoidal identity. For an arbitrary index $i$ and sequence $\ii$, we write $i \in \ii$ if $i$ appears in $\ii$. \end{notation}

Given two objects $M,N \in \Rbim$ we say they are \emph{biadjoint} if $M \ot -$ and $N \ot -$ are left and right adjoints of each other, and the same for $- \ot M$ and $- \ot N$. If $M$
and $N$ are biadjoint, so are $M\{1\}$ and $N\{-1\}$. We often want to specify additional compatibility between various adjunction maps, but we pass over the details here (see
\cite{Lau} for more information on biadjunction).

\begin{prop} Each object in $\mHC$ (resp. $\mHC_1$, $\mHC_2$) has a biadjoint, and $B_i$ is self-biadjoint. Let $\omega$ be the $t$-antilinear anti-involution on $\mH$ which fixes
$b_i$, i.e. $\omega(t^ab_{\ii})=t^{-a}b_{\sigma(\ii)}$ where $\sigma$ reverses the order of a sequence. There is a contravariant functor on $\mHC$ sending an object to its biadjoint,
and it descends on the Grothendieck ring to $\omega$. \end{prop}

\begin{defn} An \emph{adjoint pairing} on $\mH$ is a pairing where each $b_i$ is self-adjoint, so that $(x,b_iy)=(b_ix,y)$ and $(x,yb_i)=(xb_i,y)$ for all $x,y \in \mH$ and all $i \in
I$. Equivalently, for any $m \in \mH$, $(mx,y)=(x,\omega(m)y)$ and $(xm,y)=(x,y\omega(m))$. \end{defn}

There is a bijection between adjoint pairings $(,)$ and traces $\epsilon$, defined by letting $(x,y)=\epsilon(\omega(x)y)$, or conversely $\epsilon(y)=(1,y)$. Adjoint pairings appear
often in the literature, for instance \cite{Mur} (although they are usually $\Ztt$-linear in both factors, unlike our current semilinear definition). Semilinear adjoint pairings will be
crucially important, due to the following remark.

\begin{remark} \label{potentialcatfnH} Let $\mC$ be a monoidal category with objects $B_i$, such that $B_i$ are self-biadjoint. We assume that $\mC$ is additive and graded and has
isomorphisms (\ref{dc-ii})-(\ref{dc-ipi}). We call such a category a \emph{potential categorification} of $\mH$. In this case, there is a map of rings from $\mH$ to $[\mc{C}]$ sending
$b_i$ to $[B_i]$, and (under suitable finite-dimensionality conditions) we get an adjoint semilinear pairing on $\mH$ via $(b_{\ii},b_{\jj}) = \gdim \HOM_{\mc{C}}(B_{\ii},B_{\jj}) \in
\Zttt$, the graded dimension as a vector space. Denote the pairing and its associated trace map as $(,)_{\mC}$ and $\epsilon_{\mC}$.

Instead, we may assume $\mC$ is an enriched monoidal subcategory, containing objects $B_i$. The isomorphisms (\ref{dc-ii})-(\ref{dc-ipi}) typically have no meaning in this context, since
there are no grading shifts or direct sums, but we can require that they \emph{Yoneda-hold}, that is, they hold after the application of any $\Hom(-,X)$ functor (to graded vector spaces).
There is no definition of a Grothendieck ring in this case, but we still get an induced adjoint semilinear pairing induced by Hom spaces. We call this an \emph{enriched potential
categorification}. \end{remark}

We may use pairings to distinguish between different potential categorifications. The next proposition allows us to specify the pairing induced by a categorification by only
investigating certain $\HOM$ spaces.

\begin{prop} \label{prop-tracesincreasing} Traces on $\mH$ are uniquely determined by their values $\epsilon(b_{\ii})$ on increasing monomials $\ii$. Equivalently, adjoint pairings are
determined by $(1,b_{\ii})$ for $\ii$ increasing. If $\ii$ is non-repeating and $\jj$ is a permutation of $\ii$, then $\epsilon(b_{\ii})=\epsilon(b_{\jj})$. \end{prop}

We quickly sketch the proof. Moving an index from the beginning of a monomial to the end, or vice versa, will be called \emph{cycling} the monomial. It is clear, using biadjointness or
the definition of trace, that the value of $\epsilon$ is invariant under cycling. It is not difficult to show that any monomial in $W$ (in the letters $s_i$) will reduce, using the
Coxeter relations and cycling, to an increasing monomial. When the monomial is already non-repeating, one need only use cycling and $s_is_j=s_js_i$ for $i,j$ distant. Finally, using
induction on the length of the monomial, the same principle shows that any monomial in $\mH$ reduces to a linear combination of increasing monomials, and therefore $\epsilon$ is
determined by these.

The upshot is that, given a potential categorification, one knows the dimension of all $\HOM(B_{\ii},B_{\jj})$ so long as one knows the dimension of $\HOM(B_{\emptyset},B_{\ii})$ for
$\ii$ increasing. Note that not every choice of $(1,b_{\ii})$ for all increasing $\ii$ will yield a well-defined trace map.

Consider the adjoint pairing given by $\epsilon_{\rm{std}}(b_{\ii})=(1,b_{\ii})=t^d$ for $\ii$ non-repeating of length $d$. This is the
semilinear version of the pairing found in \cite{Mur}, which picks out the coefficient of the identity in the standard basis of $\mH$, and is
called the \emph{standard pairing}. Soergel showed that $\HOM(B_{\ii},B_{\jj})$ is a free graded left (or right) $R$-module of rank
$(b_{\ii},b_{\jj})$ using this pairing. In particular, for $\ii$ increasing, $\HOM(R,B_{\ii})$ is generated by a single element in degree
$d(\ii)$. Since the graded dimension of $R$ is $\frac{1}{(1-t^2)^n}$ we have that $(1,b_{\ii})_{\mHC}=\frac{t^d}{(1-t^2)^n}$ is a rescaling of
the standard pairing.

Now let $\epsilon$ be the quotient map $\mH \to \Ztt$ by the ideal generated by all $b_i$. It is a homomorphism to a commutative algebra, so it is a trace. The corresponding pairing
satisfies $(1,1)=1$ and $(x,y)=0$ for monomials $x,y$ if either monomial is not $1$. We call this the \emph{trivial pairing}, $\epsilon_{\rm{triv}}$.

%
\subsection{The Temperley-Lieb Algebra}
\label{subsec-temperley}
%

Here again we state without proof some basic facts about Temperley-Lieb algebras. They were originally defined
by Temperley and Lieb in \cite{TL}, and were given a topological interpretation by Kauffman \cite{Kau}. There
are many good expositions for the topic, such as \cite{FG,Westbury}.

\begin{defn} The \emph{Temperley-Lieb algebra} $\mTL$ is the $\Ztt$-algebra generated by $u_i, i \in I$ with relations

\begin{eqnarray}
u_i^2 & = & \qtwo u_i \label{eqn-uisq}\\
u_iu_j & = & u_ju_i \ \mathrm{for}\  |i-j|\ge 2 \label{eqn-uiuj}\\
u_iu_ju_i & = & u_i \label{eqn-TLipi}  \textrm{ for adjacent } i, j.
\end{eqnarray} 
\end{defn}

\begin{prop} For $i,j \in I$ adjacent, consider the element of $\mc{H}$ defined by $c_{ij} \define b_ib_jb_i - b_i = b_jb_ib_j - b_j$, where the equality arises from relation (\ref{eqn-bibpbi}). There is a
surjective map $\mH \to \mTL$ sending $b_i$ to $u_i$ for all $i \in I$, and whose kernel is generated by $c_{ij}$ for $i,j \in I$ adjacent. \end{prop}

Once again, write $u_{\ii}$ for a monomial in the above generators, with all the same conventions as before. The map $\omega$ descends from $\mH$ to $\mTL$, and we define an
\emph{adjoint pairing} on $\mTL$ in the same way, with $u_i$ replacing $b_i$ everywhere. The results of Proposition \ref{prop-tracesincreasing} apply equally to $\mTL$.

\begin{defn} \label{potentialcatfnTL}  A category $\mc{C}$ as in Remark \ref{potentialcatfnH} is a \emph{potential categorification} of $\mTL$ if it has objects $U_i$ satisfying
	
\begin{eqnarray}
	U_i \otimes U_i & \cong & U_i\{1\} \oplus U_i\{-1\} \label{tl-ii} \\
	U_i \otimes U_j & \cong & U_j \otimes U_i \label{tl-ij} \textrm{ for distant } i, j \\
	U_i \otimes U_j \otimes U_i & \cong & U_i \label{tl-ipi} \textrm{ for adjacent } i, j. 
\end{eqnarray}

We call it an \emph{enriched potential categorification} if it is an enriched category with objects $U_i$ such that these isomorphisms Yoneda-hold. \end{defn}

A permutation $\sigma \in S_{n+1}$ is called \emph{321-avoiding} if it never happens that, for $i<j<k$, $\sigma(i)>\sigma(j)>\sigma(k)$. It turns out that, using the Temperley-Lieb
relations, every monomial $u_{\jj}$ is equal to a scalar times some $u_{\ii}$ where $\ii$ is \emph{321-avoiding}, i.e. if viewed as a word in the symmetric group it represents a
\emph{reduced} expression for a 321-avoiding permutation. Moreover, between 321-avoiding monomials, the only further relations come from (\ref{eqn-uiuj}), and hence it is easy to pick out
a basis from this spanning set. See \cite{FG} for more details.

The Temperley-Lieb algebra has a well-known topological interpretation where an element of $\mTL$ is a linear
combination of crossingless matchings (isotopy classes of embedded planar 1-manifolds) between $n+1$ bottom
points and $n+1$ top points. Multiplication of crossingless matchings consists of vertical concatenation (where
$ab$ is $a$ above $b$), followed by removing any circles and replacing them with a factor of $\qtwo$. In this
picture, $u_i$ becomes the following:

\igc{.6}{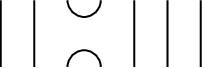}

The basis of 321-avoiding monomials agrees with the basis of crossingless matchings. Any increasing monomial is 321-avoiding. Increasing monomials are easy to visualize topologically, as
they have only ``right waves" and ``simple cups and caps." For example:

\begin{center} $u_1u_2u_3u_6u_7u_9 \mapsto \igv{.6}{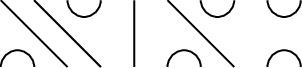}$ \end{center}

As an example of a monomial which is not increasing:

\begin{center} $u_4u_3u_1u_2 \mapsto \igh{.6}{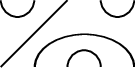}$ \end{center}

Given a crossingless matching, its \emph{closure} is a configuration of circles in the punctured plane obtained
by wrapping the top boundary around the puncture to close up with the bottom boundary, as in Figure \ref{closure}.  Circle
configurations have two topological invariants: the number of circles and the \emph{nesting number}, which is
the number of circles which surround the puncture, and is equal to $n+1-2l \ge 0$ for some $l \ge 0$. Given a
scaling factor for each possible nesting number, one constructs a trace by letting $\epsilon(u_{\ii}) = c_k
\qtwo^m$ where $m$ is the number of circles in the closure of $u_{\ii}$ and $c_k$ is the scaling factor
associated to its nesting number $k$. To calculate $(x,y)$, we place $y$ below an upside-down copy of $x$ (or vice
versa), and then take the closure. All pairings/traces on $\mTL$ can be constructed this way, so they are all
topological in nature.

\begin{figure}
	\label{closure}
$\ig{.6}{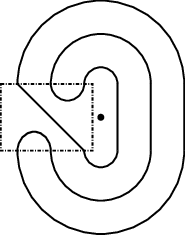}$
\caption{An example of the closure of a crossingless matching}
\end{figure}

The Temperley-Lieb algebra has a \emph{standard pairing} of its own for which $c_k=1$ for all nesting numbers $k$: $\epsilon_{\rm{std}}(u_{\ii})=\qtwo^m$ as above. One can check that
$\epsilon_{\rm{std}}(u_{\ii})=\qtwo^{n+1-d(\ii)}$ for an increasing monomial. This is \emph{not} related to the standard pairing on $\mH$, which does not descend to $\mc{TL}$. On the
other hand, $\epsilon_{\rm{triv}}$ clearly does descend to a pairing \emph{trivial pairing} on $\mc{TL}$, which only evaluates to a non-zero number when the nesting number is $n+1$.

It turns out that the pairing on $\mTL$ arising from our categorification will satisfy $(1,1)=\frac{t^n}{(1-t^2)}\qtwo^n - \frac{t^2}{(1-t^2)}$ and
$(1,u_{\ii})=\frac{t^n}{(1-t^2)}\qtwo^{n-d}$. We will call the associated trace $\epsilon_{\rm{cat}}$. Clearly $\epsilon_{\rm{cat}}= \frac{t^n}{(1-t^2)\qtwo} \epsilon_{\rm{std}} -
\frac{t^2}{(1-t^2)}\epsilon_{\rm{triv}}$. In particular, on any monomial $x \ne 1$, our trace will agree with a rescaling of the standard trace. When $n=1$, the algebras $\mTL$ and $\mH$
are already isomorphic, and $\epsilon_{\rm{cat}}$ agrees with the rescaling of the standard trace on $\mH$ discussed in the end of Section \ref{subsec-hecke}.

%
\subsection{Definition of Soergel diagrammatics}
\label{subsec-diagrammatics}
%

We now give a diagrammatic description of the category $\mHC_1$, as discovered in \cite{EKh}. Since the category to be defined will be equivalent to the category of Bott-Samelson
bimodules, we will abuse notation temporarily and use the same names.

\begin{defn} \label{defn-planargraph} In this paper, a \emph{planar graph in the strip} is a finite graph with boundary $(\Gamma, \partial \Gamma)$ embedded in $(\R \times [0,1],\R
\times \{0,1\})$. In other words, all vertices of $\Gamma$ occur in the interior $\R \times (0,1)$, and removing the vertices we have a 1-manifold with boundary whose
intersection with $\R \times \{0,1\}$ is precisely its boundary. This allows edges which connect two vertices, edges which connect a vertex to the boundary, edges which connect two
points on the boundary, and edges which form circles (closed 1-manifolds embedded in the plane).

We generally refer to $\R \times \{0,1\}$ as the \emph{boundary}, which consists of two components, the \emph{top boundary} $\R \times \{1\}$ and the \emph{bottom boundary} $\R \times
\{0\}$. We refer to a local segment of an edge which hits the boundary as a \emph{boundary edge}; there is one boundary edge for each point on the boundary of the graph. We use the
word \emph{component} to mean a connected component of a graph with boundary.

This definition clearly extends to other subsets of the plane with boundary, so that we can speak of planar graphs in a disk or planar graphs in an annulus. The annulus has two
boundary components, \emph{inner} and \emph{outer}. When we do not specify, we always mean a planar graph in the strip. \end{defn}

\begin{figure}
	\caption{An example of a planar graph in the strip, with colored edges}
	$\ig{.8}{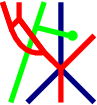}$
\end{figure}

We will be drawing morphisms in $\mHC_1$ as planar graphs with edges labelled in $I$. Instead of putting labels everywhere, we color the edges, assigning a color to each index in $I$.
Henceforth, we use the term ``color" and ``index" interchangeably.

We now define $\mHC_1$ anew. Let $\mHC_1$ be the monoidal category, with hom spaces enriched over graded vector spaces, which is defined as follows.

\begin{defn} An object in $\mHC_1$ is given by a sequence of indices $\ii$, which is visualized as $d$ points on the real line $\R$, labelled or ``colored'' by the indices in order
from left to right. These objects are also called $B_{\ii}$. The monoidal structure on objects is concatenation of sequences. \end{defn}

\begin{defn} \label{defn-morphhc} Consider the set of \emph{isotopy classes} of planar graphs in the strip whose edges are colored by indices in $I$ such that only four types of
vertices exist: univalent vertices or ``dots'', trivalent vertices with all three adjoining edges of the same color, 4-valent vertices whose adjoining edges alternate in colors between
$i$ and $j$ distant, and 6-valent vertices whose adjoining edges alternate between $i$ and $j$ adjacent. This set has a grading, where the degree of a graph is +1 for each dot and -1
for each trivalent vertex; $4$-valent and $6$-valent vertices are of degree $0$. The allowable vertices, which we call ``generators," are pictured here: \igc{1}{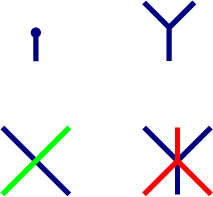}

The intersection of a graph with the boundary yields two sequences of colored points on $\R$, the top boundary $\ii$ and the bottom boundary
$\jj$. In this case, the graph is viewed as a morphism from $\jj$ to $\ii$. For instance, if ``blue" corresponds to the index $i$ and ``red" to
$j$, then the lower right generator is a degree 0 morphism from $jij$ to $iji$. Although this paper is easiest to read in color, it should be
readable in black and white: the colors appearing are typically either blue, red, green, or miscellaneous and irrelevant. We use the convention
throughout that blue (the darker color) is always adjacent to red (the middle color) and distant from green (the lighter color).

We let $\Hom_{\mHC_1}(B_{\ii},B_{\jj})$ be the graded vector space with basis given by planar graphs as above which have the correct top and bottom boundary, modulo relations
(\ref{assoc1}) through (\ref{assoc3}). As usual in a diagrammatic category, composition of morphisms is given by vertical concatenation (read from bottom to top), the monoidal
structure is given by horizontal concatenation, and relations are to be interpreted monoidally (that is, they may be applied locally inside any other planar diagram). \end{defn}

The relations are given in terms of colored graphs, but with no explicit assignment of indices to colors. They hold for \emph{any} assignment of indices to colors, so long as certain
adjacency conditions hold. We will specify adjacency for all pictures, although one can generally deduce it from the fact that 6-valent vertices only join adjacent colors, and 4-valent
vertices only join distant colors.

For example, these first four relations hold, with blue representing a generic index.

\begin{equation} \label{assoc1} \ig{.9}{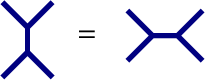} \end{equation}	
\begin{equation} \label{unit} \ig{.9}{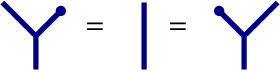} \end{equation}
\begin{equation} \label{needle} \ig{.9}{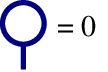} \end{equation}
\begin{equation} \label{dotslidesame} \ig{.9}{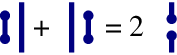} \end{equation}

We will repeatedly call a picture looking like (\ref{needle}) by the name ``needle." Note that a needle is not necessarily zero if there is something in the interior. Note that a
circle is just a needle with a dot attached, by (\ref{unit}), so that an empty circle evaluates to $0$.

\begin{remark} \label{treereduction} It is an immediate consequence of relations (\ref{assoc1}) and (\ref{unit}) that any \emph{tree} (connected graph with boundary without cycles) of one
color is equal to: \begin{itemize} \item If it has no boundary, two dots connected by an edge. Call the entire component a \emph{double dot}. \item If it has one boundary edge, a single
dot connected by the edge to the boundary. Call the component a \emph{boundary dot}. \item If it has more boundary edges, a tree with no dots and the fewest possible number of trivalent
vertices needed to connect the boundaries. Moreover, any two such trees are equal. Call the component a \emph{simple tree}. \end{itemize} We refer to this as \emph{tree reduction}.

This applies only to components of a graph which are a single color. Even if the blue part of a graph looks like a tree, if other colors overlap then we may not apply tree reduction in
general. \end{remark}

\begin{figure}
	$\ig{.8}{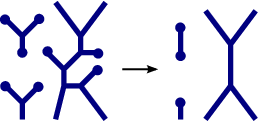}$
	\caption{An example of tree reduction}
\end{figure}

In the following relations, the two colors are distant.

\begin{equation} \label{R2} \ig{.8}{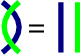} \end{equation}
\begin{equation} \label{distslidedot} \ig{1}{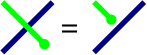} \end{equation}
\begin{equation} \label{distslide3} \ig{.8}{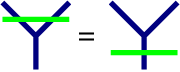} \end{equation}
\begin{equation} \label{dotslidefar} \ig{.8}{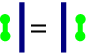} \end{equation}

In this relation, two colors are adjacent, and both distant to the third color.

\begin{equation} \label{distslide6} \ig{.8}{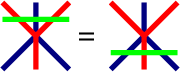} \end{equation}

In this relation, all three colors are mutually distant.

\begin{equation} \label{distslide4} \ig{1}{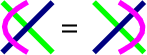} \end{equation}

\begin{remark} Relations (\ref{R2}) thru (\ref{distslide4}) indicate that any part of the graph colored $i$ and any part of the graph colored $j$ ``do not
interact'' for $i$ and $j$ distant. That is, one may visualize sliding the $j$-colored part past the $i$-colored part, and it will not change the morphism. We call
this the \emph{distant sliding property}. \end{remark}

In the following relations, the two colors are adjacent.

\begin{equation} \label{dot6} \ig{.8}{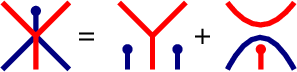} \end{equation}
\begin{equation} \label{ipidecomp} \ig{.8}{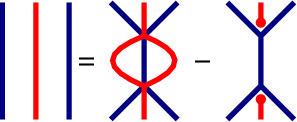} \end{equation}
\begin{equation} \label{assoc2} \ig{.8}{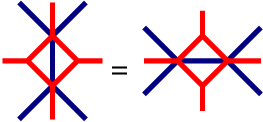} \end{equation}
\begin{equation} \label{dotslidenear} \ig{1}{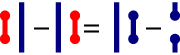} \end{equation}

In this final relation, the colors have the same adjacency as $\{1,2,3\}$.

\begin{equation} \label{assoc3} \ig{1}{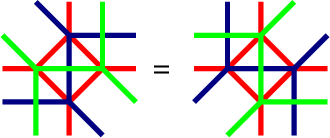} \end{equation}

This concludes the list of relations defining $\mHC_1$.

\begin{remark} We chose here to describe $\mHC_1$ in terms of planar graphs with relations, with the notion of isotopy built-in, rather than in terms of generators and relations. Note
however that using isotopy and (\ref{unit}) we get $\ig{.4}{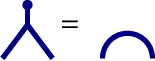}$. Therefore, all ``cups" and ``caps" can be expressed in terms of the generators. By adding new relations corresponding to
isotopy, one could give a presention of the category where the ``generators" above (and their isotopy twists) are really generators. This is how the category is presented in
\cite{EKh}.\end{remark}

We will occasionally use a shorthand to represent double dots. We identify a double dot colored $i$ with the polynomial $f_i\in R$, and to a linear combination of
disjoint unions of double dots in the same region of a graph, we associate the appropriate linear combination of products of $f_i$. For any polynomial $f\in R$, a
square box with a polynomial $f$ in a region will represent the corresponding linear combination of graphs with double dots.

\plabel{$f_i^2f_j$}{48}{12}
For instance, $\ig{1.5}{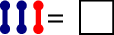}$.

Relations (\ref{dotslidesame}), (\ref{dotslidenear}), and (\ref{dotslidefar}) are referred to as \emph{dot forcing rules}, because they describe at what price one can ``force" a double
dot to the other side of a line. The three relations imply that, given a line and an arbitrary collection of double dots on the left side of that line, one can express the morphism as
a sum of diagrams where all double dots are on the right side, or where the line is ``broken" (as illustrated next). Rephrasing this, for any polynomial $f$ there exist
polynomials $g$ and $h$ such that

\plabelthree{$f$}{288}{816}{$g$}{344}{816}{$h$}{392}{816} \begin{equation} \ig{1.2}{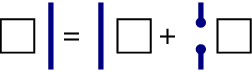} \label{polyslidegen} \end{equation}

The polynomials appearing can in fact be found using the Demazure operator $\partial_i$, and in particular, $h=\partial_i(f)$. One particular implication is that
\plabeltwo{$f$}{288}{816}{$f$}{344}{816} \begin{equation} \ig{1.2}{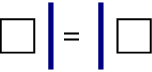} \end{equation} whenever $f$ is a polynomial invariant under $s_i$ (and blue represents $i$). As an
exercise, the reader can check that $f_i^2$ slides through a line colored $i$. These polynomial relations are easy to deduce, or one can refer to \cite{EKh} (see p.7, p. 16-17, and
relation 3.16).

We have an bimodule action of $R$ on morphisms by placing boxes (i.e. double dots) in the leftmost or rightmost regions of a graph. Now we can formulate the main result of \cite{EKh}.

\begin{thm} \label{mainthmekh} There is a functor from this diagrammatic category $\mHC_1$ to the earlier definition in terms of Bott-Samelson bimodules. This functor sends $\ii$ to
the bimodule $B_{\ii}$ and a planar graph to a map of bimodules, preserving the grading and the $R$-bimodule action on morphisms. This functor is an equivalence of categories.
\end{thm}

\begin{cor} \label{freemodules} The $R$-bimodules $\Hom_{\mHC_1}(B_{\ii},B_{\jj})$ are free as left (or right) $R$-modules. In other words, placing double dots to the left of a graph
is a torsion-free operation. \end{cor}

Now we have justified our abuse of notation. In this paper, we will never need to know explicitly what map of $R$-bimodules a planar graph corresponds to, so the interested reader can
see \cite{EKh} for details. In fact, we will not use Theorem \ref{mainthmekh} at all, preferring to work entirely with planar graphs. However, we do use Corollary \ref{freemodules}, a
fact which would be difficult to prove diagrammatically.

The proof of Theorem \ref{mainthmekh} can be quickly summarized: first, one explicitly constructs a functor from the diagrammatic category to the Bott-Samelson category. Then, using
the observations of the next section, one shows that the diagrammatic category is a potential categorification of $\mH$, and that the diagrammatic category, the Bott-Samelson
category, and the image of the former in the latter all induce the same adjoint pairing on $\mH$. Therefore the functor is fully faithful.

%
\subsection{Understanding Soergel diagrammatics}
\label{subsec-diagrammatics2}
%

Let us explain diagrammatically why the category $\mHC_1$ is a potential categorification of $\mH$, and induces the aforementioned adjoint pairing.

\begin{defn} \label{gradedkaroubi} Given a category $\mc{C}$ whose morphism spaces are $\Z$-modules, we may take its \emph{additive closure}, which formally adds direct sums of
objects, and yields an additive category. Given $\mc{C}$ whose morphism spaces are graded $\Z$-modules, we may take its \emph{grading closure}, which formally adds shifts of objects,
but restricts morphisms to be homogeneous of degree 0. Given $\mc{C}$ an additive category, one may take the \emph{idempotent completion} or \emph{Karoubi envelope}, which formally
adds direct \emph{summands}. Recall that the Karoubi envelope has as objects pairs $(B,e)$ where $B$ is an object in $\mc{C}$ and $e$ an idempotent endomorphism of $B$. This object
acts as though it were the ``image'' of this projection $e$, and behaves like a direct summand. When taking the Karoubi envelope of a graded category (or a category with graded
morphisms) one restricts to homogeneous degree 0 idempotents. We refer in this paper to the entire process which takes a category $\mc{C}$, whose morphism spaces are graded
$\Z$-modules, and returns the Karoubi envelope of its additive and grading closure as taking the \emph{graded Karoubi envelope}. All these transformations interact nicely with monoidal
structures. For more information on Karoubi envelopes see \cite{BNM}.

We let $\mHC_2$ be the graded additive closure of $\mHC_1$, and $\mHC$ be the graded Karoubi envelope of $\mHC_1$.\end{defn}

We wish to show that the isomorphisms (\ref{dc-ii}) through (\ref{dc-ipi}) hold in $\mHC_2$. Relation (\ref{R2}) immediately implies that $B_i \ot B_j \cong B_j \ot B_i$ for $i,j$
distant, with the isomorphism being given by the 4-valent vertex.

We have the following equality:
\begin{equation} \label{iidecomp} \ig{1}{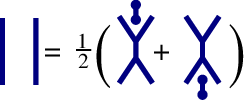}. \end{equation}
To obtain this, use (\ref{unit}) to stretch two dots from the two lines into the middle, and then use (\ref{dotslidesame}) to connect them. The identity $\id_{ii}$ decomposes as a sum
of two orthogonal idempotents, each of which is the composition of a ``projection" and an ``inclusion" map of degree $\pm 1$, to and from $B_i$ (explicitly, $\id_{ii}=i_1p_1 + i_2p_2$
where $p_1i_1=\id_i$, $p_2i_2=\id_i$, $p_1i_2=0=p_2i_1$). This implies that $B_i \ot B_i \cong B_i\{1\} \oplus B_i\{-1\}$, and is a typical example of how direct sum decompositions
work in diagrammatic categories.

Similarly, the two color variants of relation (\ref{ipidecomp}) together express the direct sum decompositions in the Karoubi envelope \begin{eqnarray} B_i \otimes B_{i+1} \otimes B_i
= C_{ij} \oplus B_i\\ B_{i+1} \otimes B_i \otimes B_{i+1} = C_{ji} \oplus B_{i+1}. \end{eqnarray} Again, the identity $\id_{i(i+1)i}$ is decomposed into orthogonal idempotents. The
second idempotent factors through $B_i$, and the corresponding object in the Karoubi envelope will be isomorphic to $B_i$. The first idempotent, which we call a ``doubled 6-valent
vertex," corresponds to a new object $C_{ij}$ in the idempotent completion. It turns out that the doubled 6-valent vertex $C_{ij}$ for ``blue red blue" is isomorphic in the Karoubi
envelope to the doubled 6-valent vertex $C_{ji}$ for ``red blue red" (i.e. their images are isomorphic). We may abuse notation and call both these new objects $C_{ij}$; it is a summand
of both $i(i+1)i$ and $(i+1)i(i+1)$. The image of $C_{ij}$ in the Grothendieck group is $c_{ij}$.

We can also understand the induced pairing on $\mH$ using diagrammatic arguments. The theorems below are proven in \cite{EKh}, and we will not use them in this paper (except
motivationally), proving their analogs in the Temperley-Lieb case directly.

\begin{thm} \label{colorreduction} {\bf (Color Reduction)} Consider a morphism $\phi \colon \emptyset \to \ii$, and suppose that the index $i$ (blue) appears in $\ii$ zero times
(respectively: once). Then $\phi$ is in the $\Bbbk$-span of graphs which only contain blue in the form of double dots in the leftmost region of the graph (respectively: as well as a
single boundary dot). This result may be obtained simultaneously for multiple indices $i$. \end{thm}

\begin{cor} \label{cor-homspaces} The space $\Hom_{\mHC_1}(\emptyset,\emptyset)$ is precisely the graded ring $R$. In other words, it is freely generated
(over double dots) by the empty diagram. The space $\Hom_{\mHC_1}(\emptyset,\ii)$ for $\ii$ non-repeating is a free left (or right) $R$-module of rank 1, generated by the following
morphism of degree $d(\ii)$.

\igc{1}{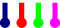}
\end{cor}

The proof of the theorem does not use any sophisticated technology, only convoluted pictorial arguments. It comprises the bulk of \cite{EKh}. The corollary implies that
$\epsilon_{\mHC_1}(b_{\ii})=\frac{t^d}{(1-t^2)^n}$ for $\ii$ non-repeating of length $d$, as stated in Section \ref{subsec-hecke}.

%
\subsection{Aside on Karoubi Envelopes and Quotients}
\label{subsec-karoubi}
%

Return to the setup of Definition \ref{gradedkaroubi}. If $\mc{C}$ is a full subcategory of (graded) $R$-bimodules for some ring $R$, then the transformations described above behave as
one would expect them to. In particular, the Karoubi envelope agrees with the full subcategory which includes all summands of the previous objects. The Grothendieck group of the
Karoubi envelope is in some sense ``under control," if one understands indecomposable $R$-bimodules already. On the other hand, the Karoubi envelope of an arbitrary additive category
may be enormous, and to control the size of its Grothendieck group one should understand and classify all idempotents in the category, a serious task. Also, arbitrary additive
categories need not have the Krull-Schmidt property, making their Grothendieck groups even more complicated.

The Temperley-Lieb algebra is obtained from the Hecke algebra by setting the elements $c_{ij}$ to zero, for $i=1, \ldots, n-1$. These elements lift in the Soergel
categorification to objects $C_{ij}$. The obvious way one might hope to categorify $\mTL$ would be to take the quotient of the category $\mHC$ by each object $C_{ij}$.

To mod out an additive \emph{monoidal} category $\mc{C}$ by an object $Z$, one must kill the \emph{monoidal} ideal of $\id_Z$ in $\Mor(\mc{C})$. That is, the
morphism space $\Hom(X,Y)$ in the quotient category is exactly $\Hom_{\mc{C}}(X,Y)$ modulo the submodule of morphisms factoring through $V \otimes Z \otimes W$ for any
$V,W$. If the category is drawn diagrammatically, one need only kill any diagram which has $\id_Z$ as a subdiagram.

We have not truly drawn $\mHC$ diagrammatically, only $\mHC_1$. The object we wish to kill is not an object in $\mHC_1$; the closest thing we have is the
corresponding idempotent, the doubled 6-valent vertex. However, this is not truly a problem, due to the following proposition, whose proof we leave to the reader.

\begin{prop} Let $\mc{C}_1$ be an additive category, $B$ an object in $\mc{C}_1$, and $e$ an idempotent in $\End(B)$. Let $\mc{D}_1$ be the quotient of $\mc{C}_1$ by
the morphism $e$. Let $\mc{C}$ and $\mc{D}$ be the respective Karoubi envelopes. Finally, let $\mc{D}^\prime$ be the quotient of $\mc{C}$ by the identity of the
object $(B,e)$. Then there is a natural equivalence of categories from $\mc{D}$ to $\mc{D}^\prime$.

The analogous statement holds when one considers graded Karoubi envelopes. \end{prop}

\begin{remark} Note that $\mc{D}^\prime$ has more objects than $\mc{D}$, but they are still equivalent. For instance, $(B,e)$ and
$(B,0)$ are distinct (isomorphic) objects in $\mc{D}^\prime$, but are the same object in $\mc{D}$. \end{remark}

So to categorify $\mTL$, one might wish to take the quotient of $\mHC_1$ by the doubled 6-valent vertex, and then take the Karoubi envelope. This is easy to do diagrammatically, which is
one advantage to the diagrammatic approach over the $R$-bimodule approach. The quotient of $\mHC_1$ will no longer be a category which embeds nicely as a full subcategory of bimodules.
One might worry that Krull-Schmidt fails, or that to understand its Karoubi envelope one must classify all idempotents therein. Thankfully, our calculation of HOM spaces will imply easily
that its graded additive closure is Krull-Schmidt and is \emph{already} idempotent closed, so it is equivalent to its own Karoubi envelope (see Section \ref{subsec-usingpairing}).

\section{The Quotient Category $\mTLC$}
\label{sec-tlc}

%
\subsection{A motivating calculation}
\label{subsec-motivate}
%

As discussed in the previous section, our desire is to take the quotient of $\mHC_1$ by the doubled 6-valent vertex, and then take the graded Karoubi envelope.

An important consequence of relations (\ref{dot6}) and (\ref{needle}) is that
\begin{equation} \label{dot6surrounded} \ig{.8}{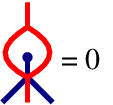} \end{equation}
from which it follows, using (\ref{ipidecomp}), that
\begin{equation} \ig{.8}{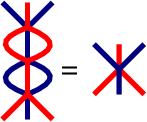} \label{tripleequalssingle} \end{equation}
so the (monoidal) ideal generated in $\mHC_1$ by a doubled 6-valent vertex is the same as the ideal generated by the 6-valent vertex.

\begin{claim} The following relations are all equivalent (the ideals they generate are equal).

\begin{equation} \ig{.8}{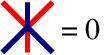} \label{AAA} \end{equation}
\begin{equation} \ig{.8}{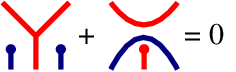} \label{reldotmod6}\end{equation}
\begin{equation} \ig{.8}{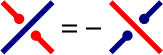} \label{switch}\end{equation}
\begin{equation} \ig{.8}{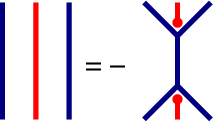}  \label{tlc-threelines}\end{equation}
\begin{equation} \ig{.8}{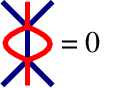} \label{EEE} \end{equation}
\end{claim}

\begin{proof}
	(\ref{AAA}) $\implies$ (\ref{reldotmod6}): Add a dot, and use relation (\ref{dot6}).
	
	(\ref{reldotmod6}) $\implies$ (\ref{switch}): Add a dot to the top, and use (\ref{unit}).
	
	(\ref{switch}) $\implies$ (\ref{reldotmod6}): Apply to the middle of the diagram.
	
	(\ref{switch}) $\implies$ (\ref{tlc-threelines}): Stretch dots from the blue strands towards the red strand using (\ref{unit}), and then apply (\ref{switch}) to the middle.
	
	(\ref{tlc-threelines}) $\implies$ (\ref{EEE}): Use relation (\ref{ipidecomp}).
	
	(\ref{EEE}) $\implies$ (\ref{AAA}): Use (\ref{tripleequalssingle}).
\end{proof}

Modulo 6-valent vertices, the relations (\ref{dot6}) and (\ref{ipidecomp}) become (\ref{reldotmod6}) and (\ref{tlc-threelines}) above. All other relations involving
6-valent vertices, namely (\ref{assoc2}), (\ref{assoc3}), and (\ref{distslide6}), are sent to zero modulo 6-valent vertices. Relation (\ref{switch}) implies both
(\ref{reldotmod6}) and (\ref{tlc-threelines}) without reference to any graphs using 6-valent vertices. So if we wish to rephrase our quotient in terms of graphs that
never have 6-valent vertices, the sole necessary relation imposed by the fact that 6-valent vertices were sent to zero is the relation (\ref{switch}).

Suppose we only allow ourselves univalent, trivalent, and 4-valent vertices, but no 6-valent vertices, in a graph $\Gamma$. Then the $i$-graph of $\Gamma$, which
consists of all edges colored $i$ and all vertices they touch, will be disjoint from the $i+1$- and $i-1$-graphs of $\Gamma$. The distant sliding property implies
that the $i$-graph and the $j$-graph of $\Gamma$ effectively do not interact, when $i$ and $j$ are distant. This will motivate the definition in the next section.

%
\subsection{Diagrammatic definition of $\mTLC$}
\label{subsec-deftlc}
%

\begin{defn} We let $\mTLC_1$ be the monoidal category, with hom spaces enriched over graded vector spaces, defined as follows. Objects will be sequences of colored points on the
line $\R$, which we will call $\ii$ or $U_{\ii}$. Consider the set whose elements are described as follows:

\begin{enumerate}
\item For each $i \in I$, consider a planar graph $\Gamma_i$ in the strip, which is drawn with edges colored $i$ (see Definition \ref{defn-planargraph}).
\item The only vertices in $\Gamma_i$ are univalent vertices (dots) and trivalent vertices.
\item The graphs $\Gamma_i$ and $\Gamma_{i+1}$ are disjoint.  All graphs $\Gamma_i$ are pairwise disjoint on the boundary.
\item We consider isotopy classes of this data, so that one may apply isotopy to each $\Gamma_i$ individually so long as it stays appropriately disjoint.
\end{enumerate}

This set has a grading, where the degree of a graph is $+1$ for each dot and $-1$ for each trivalent vertex, and the degree of an element of this set is the sum of the degrees for each
graph $\Gamma_i$. Just as in Definition \ref{defn-morphhc}, each element of the set has a top and bottom boundary which is an object in $\mTLC$, and will be thought of as a map from
the bottom boundary to the top. We let $\Hom_{\mTLC_1}(U_{\ii},U_{\jj})$ be the graded vector space with basis given by elements of the set above with bottom boundary $\jj$ and top
boundary $\ii$, modulo the relations (\ref{assoc1}) through (\ref{dotslidesame}), (\ref{dotslidenear}), and the new relation (\ref{switch}). As a reminder, the new relation is given
here again.

\begin{center}$\ig{1}{switch}$\end{center}

As before, composition of morphisms is given by vertical concatenation, the monoidal structure is given by horizontal concatenation, and relations are to be interpreted monoidally.
This concludes the definition. \end{defn}

Phrasing the definition in this fashion eliminates the need to add distant sliding rules, for these are now built into the notion of isotopy. Note that as we have stated it here,
$\Gamma_i$ and $\Gamma_j$ may have edges which are embedded in a tangent fashion, or even entirely overlap. However, such embeddings are isotopic to graph embeddings with only
transverse edge intersections, which arise as 4-valent vertices in our earlier viewpoint.

\begin{prop} The category $\mTLC_1$ is isomorphic to $\mHC_1$ modulo the 6-valent vertex. \end{prop}
	
\begin{proof} Due to the observations of Section \ref{subsec-motivate}, this is obvious. \end{proof}

Hom spaces in $\mTLC_1$ are in fact enriched over graded $R$-bimodules, by placing double dots as before. However, they will no longer be free as left or right $R$-modules, as we shall see.

\begin{remark} Note that tree reduction (see Remark \ref{treereduction}) can now be applied to any tree of a single color in $\mTLC$, regardless of what other colors are present,
since the only colors which can intersect the tree are distant colors which do not actually interfere. \end{remark}

We denote by $\mTLC$ the graded Karoubi envelope of $\mTLC_1$, and $\mTLC_2$ the graded additive closure of $\mTLC_1$. However, we will show that $\mTLC_2$ is
already idempotent-closed, so that $\mTLC_2$ and $\mTLC$ are the same.

It is obvious that 
\begin{eqnarray} 
U_i \otimes U_{i+1} \otimes U_i \cong U_i \label{tlc-ipi}\\
U_{i+1} \otimes U_i \otimes U_{i+1} \cong U_{i+1}\label{tlc-pip}
\end{eqnarray}
in $\mTLC_1$, from the relation (\ref{tlc-threelines}) and the simple calculation (using dot forcing rules) that 
\begin{equation} \ig{.8}{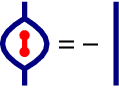}\end{equation}

For the same reasons as in Section \ref{subsec-diagrammatics2} we still have $U_i \ot U_j \cong U_j \ot U_i$ for $i,j$ distant, and $U_i \ot U_i \cong U_i\{1\} \oplus U_i\{-1\}$ in
$\mTLC_2$. Therefore $\mTLC$ is a potential categorification of $\mTL$, and induces an adjoint pairing and a trace map $\epsilon_{\mTLC}$ on $\mTL$. At this point, we have not shown that
the category $\mTLC_1$ is nonzero, so this pairing could be $0$.

%
\subsection{Using the adjoint pairing}
\label{subsec-usingpairing}
%

\begin{prop} \label{prop-traceisenough} Let $\mC_1$ be an enriched category which is a potential categorification of $\mTL$, whose objects are $U_{\ii}$ for sequences $\ii$. Let
$\mC_2$ be its additive graded closure, and $\mC$ be its graded Karoubi envelope. Suppose that the induced trace map $\epsilon_{\mC_1}$ on $\mTL$ is equal to $\epsilon_{\rm{cat}}$. Then
the set of $U_{\ii}\{n\}$ for $n \in \Z$ and $\ii$ 321-avoiding forms an exhaustive irredundant list of indecomposables in $\mC_2$. In addition, $\mC_2$ is Krull-Schmidt
and idempotent-closed (so $\mC_2$ and $\mC$ are equivalent), and $\mC$ categorifies $\mTL$. \end{prop}

This proposition is an excellent illustration of the utility of the induced adjoint pairing. We prove it in a series of lemmas, which all assume the hypotheses above.

\begin{lemma} \label{lemma-something} The object $U_{\ii}$ in $\mC_1$ has no non-trivial (homogeneous) idempotents when $\ii$ is 321-avoiding. Moreover, if both $\ii$ and $\jj$ are
321-avoiding, then $U_{\ii} \cong U_{\jj}\{m\}$ in $\mC_2$ if and only if $m=0$ and $u_{\ii}=u_{\jj}$ in $\mTL$. \end{lemma}

\begin{proof}
  Two 321-avoiding monomials in $\mTL$ are equal only if they are related by the relation (\ref{eqn-uiuj}). Since this lifts to an isomorphism $U_i \otimes U_j \cong
  U_j \otimes U_i$ in $\mC_2$, we have $u_{\ii}=u_{\jj} \implies U_{\ii} \cong U_{\jj}$.

  If an object has a 1-dimensional space of degree 0 endomorphisms, then it must be spanned by the identity map, and there can be no non-trivial idempotents. If an
  object has endomorphisms only in non-negative degrees, then it can not be isomorphic to any nonzero degree shift of itself. If two objects $X$ and $Y$ are such that
  both $\Hom(X,Y)$ and $\Hom(Y,X)$ are concentrated in strictly positive degrees, then no grading shift of $X$ is isomorphic to $Y$, since there can not be a degree
  zero map in both directions.

  Therefore, we need only show that (for 321-avoiding monomials) $U_{\ii}$ has endomorphisms concentrated in non-negative degree, with a 1-dimensional degree 0 part, and that when
  $u_{\ii} \ne u_{\jj}$, $\Hom(U_{\ii},U_{\jj})$ is in strictly positive degrees. This question is entirely determined by the pairing on $\mTL$, since it only asks about the graded
  dimension of Hom spaces.

  When $\ii$ is empty, we already know that $(1,1)=\frac{t^n}{(1-t^2)}\qtwo^n - \frac{t^2}{(1-t^2)}$, which has degree 0 coefficient $1$, and is concentrated in
  non-negative degrees.

  We know how to calculate $(x,y)$ in $\mTL$ when $x$ and $y$ are monomials, and either $x$ or $y$ is not $1$ (see Section \ref{subsec-temperley}). We draw $x$ as a
  crossingless matching, draw $y$ upside-down and place it below $x$, and close off the diagram: if there are $m$ circles in the diagram, then $(x,y)=\frac{t^n\qtwo^{m-1}}{1-t^2}$. In
  particular, if $m=n+1$ then the Hom space will be concentrated in non-negative degrees, with 1-dimensional degree 0 part. If $m<n+1$ then the Hom space will be concentrated in
  strictly positive degrees.

  We leave it as an exercise to show that, if $x$ is a crossingless matching (i.e. a 321-avoiding monomial) then the closed diagram for $(x,x)$ has exactly $n+1$
  circles. The following example makes the statement fairly clear, where $\tilde{x}$ is $x$ upside-down:

\labellist
\small\hair 2pt
\pinlabel $x$ at 280 468
\pinlabel $\tilde{x}$ at 282 434
\endlabellist

\igc{1}{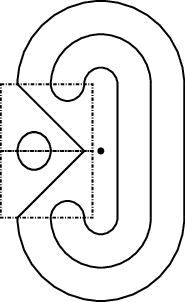}

  In this example $x$ has all 3 kinds of arcs which appear in a crossingless matching: bottom to top, bottom to bottom, and top to top. Each of these corresponds to a single
  circle in the diagram closure.

  Similarly, there are fewer than $n+1$ circles in the diagram for $(x,y)$ whenever the crossingless matchings $x,y$ are non-equal. Consider the diagram above but with the region $x$ removed. One can see that no circles are yet completed, and each boundary point of $x$'s region is matched to another by an arc. The
  number of circles is maximized when you pair these boundary points to each other, and this clearly gives the matching $x$. For any other matching $y$, two arcs will become joined
  into one, and fewer than $n+1$ circles will be created. \end{proof}

\begin{lemma} $\mC_2$ is idempotent-closed, and its indecomposables can all be expressed as grading shifts of $U_{\ii}$ for $\ii$ 321-avoiding. It has the Krull-Schmidt property.
\end{lemma}

\begin{proof} Since the Temperley-Lieb relations allow one to reduce a general word to a 321-avoiding word, one can show that every $U_{\ii}$ is isomorphic to a direct sum of shifts of
$U_{\jj}$ for 321-avoiding $\jj$, using isomorphisms and direct sum decompositions instead of the analogous Temperley-Lieb relations. Clearly these shifted $U_{\jj}$ are all
indecomposable, since they have no non-trivial idempotents; these are then all the indecomposables. Since every indecomposable in $\mC_2$ has a graded-local endomorphism ring (with
maximal ideal given by positively graded morphisms), $\mC_2$ is idempotent-closed and Krull-Schmidt (see \cite{Ring}, Section 2.2). \end{proof}

The Krull-Schmidt property implies that isomorphism classes of indecomposables form a basis for the Grothendieck group.

\begin{proof}[Proof of Proposition \ref{prop-traceisenough}] There is a $\Ztt$-linear map of rings $\mTL \to [\mC_2]$, which is evidently bijective because it sends the
321-avoiding basis to the 321-avoiding basis. Since $\mC=\mC_2$, we are done. \end{proof}

This proposition shows that Lemma \ref{mainlemma} implies Theorem \ref{maintheorem}.

\begin{remark} In analogy to the paper \cite{EKh}, the bulk of the proof of Theorem \ref{maintheorem} lies in proving that hom spaces induce a particular adjoint pairing. Beyond that
we have mostly stated the obvious. Let us note that what is obvious for $\mTL$ and $\mTLC$ is \emph{not} obvious at all when dealing with $\mH$ and $\mHC$. In particular, if we are
given a category $\mC_1$ which is a potential categorification of $\mH$ as in Proposition \ref{prop-traceisenough}, we can not conclude that $\mC$ categorifies $\mH$. We summarize the
differences here.

It is clear (for both Hecke and Temperley-Lieb) that the map $\mH \to [\mHC_2]$ is well-defined and surjective. The two main subtleties are 1) the difference between $\mHC_2$ and
$\mHC$, and 2) the injectivity of the map.

In general, one likes to examine the additive Grothendieck group only of idempotent-closed categories with the Krull-Schmidt property, because this guarantees that indecomposables form
a basis for the Grothendieck group. Thus it is convenient that $\mTLC_2$ is already idempotent-closed. Thankfully, we have a result of Soergel \cite{Soe4} that proves that $[\mHC_2]
\cong [\mHC]$, as was discussed in Remark \ref{rmk-grothofsoergel}.

To show injectivity of the map in the $\mTL$ case, we can identify a basis of $\mTL$ which is sent to a complete set of indecomposables, and then we can evaluate the trace map to show
that these indecomposables are pairwise non-isomorphic. For $\mHC$, we do not currently know what the indecomposables (i.e. idempotents) are, nor do we know their preimage in $\mH$. If
we knew a class of indecomposables which decategorified to the Kazhdan-Lusztig basis, then we could use a similar argument to the above to show that they and their shifts
form an exhaustive irredundant list of indecomposables in $\mHC$, and therefore that the map $\mH \to [\mHC]$ is injective. Soergel discusses this in the last chapter of \cite{Soe4}. This is actually a deep question, shown by Soergel
(\cite{Soe5}, see also \cite{Soe1,Soe4}) to be equivalent to proving a version of the Kazhdan-Lusztig conjectures. In any case, the result depends on the base field $\Bbbk$, and no
simple proof has been found. In particular, to prove that the graded Karoubi closure of the diagrammatic category $\mHC_1$ categorifies the Hecke algebra (for certain $\Bbbk$) we must
pass to the world of bimodules where Soergel's powerful geometric techniques will work. In particular, there is currently no proof of injectivity if one defines the category $\mHC$
diagrammatically over $\Bbbk=\Z$.

It should be emphasized that the story of $\mTL$ is a particularly easy one (as is its Kazhdan-Lusztig theory). No high-powered technical machinery is needed, and the proofs of
idempotent closure and injectivity are self-contained and diagrammatic. In fact, the arguments in this paper \emph{do} work entirely over $\Z[\frac{1}{2}]$, as can be checked. Dividing
by two must be allowed in order to split the identity of $U_{ii}$ into idempotents, as in (\ref{iidecomp}); however, it is likely that the arguments would work over $\Z$ as well.
Working over $\Z$ is discussed more extensively in \cite{EKr}.
\end{remark}

\begin{remark} A category $\mc{O}$ analog of the fact that 321-avoiding monomials lift to indecomposable Soergel bimodules, which remain
indecomposable upon passage to the Temperley-Lieb quotient, can be found in Lemma 5.2 of \cite{Str}. \end{remark}

%
\subsection{Reductions}
\label{subsec-prelimhoms}
%

When we say that a graph or a morphism ``reduces" to a set of other graphs, we mean that the morphism is in the $\Bbbk$-span of those graphs. We refer to a one-color graph, each of whose
(connected) components is either a simple tree with respect to its boundary or a double dot, as a \emph{simple forest with double dots}. If there are no double dots, it is a \emph{simple
forest without double dots}. Tree reduction implies that any graph $\Gamma_i$ without cycles reduces to a simple forest with double dots. Note also that circles in a graph are equal to
needles with a dot attached, and can be treated just like any other cycle.

If there were only one color, we could iterate the following rule (which is an implication of the dot forcing rules and (\ref{needle})) to break cycles:

\plabeltwo{$f$}{422}{366}{$\partial_{i}f$}{492}{366}
\begin{equation} \ig{1.2}{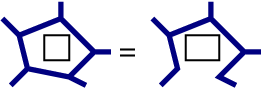} \label{polyhole}\end{equation}

We do something similar for the general case.

\begin{prop} \label{prop-reduction} In $\mTLC_1$ any morphism reduces to one where, for each $i$, the $i$-graph is a simple forest with double dots. Moreover, we may assume all double
dots are in the lefthand region. \end{prop}

\begin{figure}
	$\ig{.9}{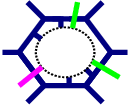}$
	\caption{An arbitrary innermost blue cycle. The dotted line encapsulates the subgraph on the interior, which may contain colors adjacent to blue.}
\end{figure}

\begin{proof} We use induction on the total number of cycles (of any color) in the graph. Suppose there is a blue colored cycle: choose one so that it delineates a single
region (i.e. there are no other cycles inside). There may be blue ``spokes" going from this cycle into the interior, but no two spokes can meet, lest they create another region. By tree
reduction on the spokes, we can assume that any blue appearing inside the cycle is in a different blue component than the cycle. Other colors may cross over the cycle, into the interior.
If we view the interior of the cycle as a graph of its own, it has fewer total cycles so we may use induction. Since the boundary of the interior contains no color blue or colors adjacent
to blue, they may be assumed to appear in the interior only in the form of double dots next to the cycle. Using dot forcing rules, we reduce to two graphs: one with the cycle broken, and
one with all these double dots on the exterior of the cycle. The former reduces by induction. For the latter, only distant colors enter the cycle, so they can be slid out of the way to
leave an empty blue cycle, which is $0$ by the rule above.

We need only do the base case, where the graph has no cycles. The dot forcing rules imply that double dots may be moved to any region of the (multicolored) graph, at the cost of breaking
a few lines. Breaking lines will never increase the number of cycles. Therefore, if we have a graph without cycles, tree reduction implies that we actually have a simple forest with
double dots, and dot forcing allows us to move these double dots to the left. The breaking of lines may require more tree reduction, yielding more double dots, but this process is finite.
\end{proof}

\begin{remark} This proposition and its proof will apply to graphs in any connected simply-connected region in the plane. \end{remark}

\begin{cor} For any $\ii$ non-repeating, $\Hom_{\mTLC_1}(\emptyset,\ii)$ is generated (as a left or right $R$-module) by a single element $\phi_{\ii}$ of degree $d(\ii)$, pictured below.
	\igc{.8}{phiii}
\end{cor}

\begin{proof} A simple forest with double dots and at most one boundary edge is no more than a boundary dot with double dots. Thus any morphism reduces to a boundary dot for each color,
accompanied by double dots. \end{proof}

To show Lemma \ref{mainlemma} we need only investigate $\Hom(\emptyset,\ii)$ for $\ii$ increasing, since we have already shown that the values of $\epsilon(u_{\ii})$ are determined by
their values for $\ii$ increasing. This space will be an $R$-bimodule where the left and right action are the same (since the lefthand and righthand regions are the same in any picture
with no bottom boundary), so we view it as an $R$-module, and we have just shown that it is cyclic. Let $I_{\ii}$ be the ideal which is the kernel of the map $R \to \HOM(\emptyset,\ii)$
sending $1 \mapsto \phi_{\ii}$; we call it the \emph{TL ideal} of $\ii$. To prove the Lemma \ref{mainlemma} is to find $I_{\ii}$ and show that the graded dimension of $R/I_{\ii}\{d\}$ is $\epsilon_{\rm{cat}}(u_{\ii})$.

\begin{remark} \label{puncturedremark} Since the space $\Hom_{\mc{HC}_1}(\emptyset,\ii)$ is a \emph{free} $R$-module, all polynomials in $I_{\ii}$ must
have arisen from reducing to some morphism which contained the relation (\ref{switch}) to a ``nice form," i.e. $\phi_{\ii}$ plus double dots. In other words, letting $\alpha_i$ be the
morphism pictured below, we want to plug $\alpha_i$ into a bigger graph, reduce it to a nice form, and see what we get.

\plabel{$\alpha_i$}{692}{248}
\igc{1}{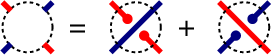}

Remember that $\alpha_i$ is actually just a 6-valent vertex with two dots attached (one red and one blue). This bigger graph, into which $\alpha_i$ is plugged, will actually be a graph on
the \emph{punctured plane} or \emph{punctured disk} with specified boundary conditions on both the outer and inner boundary. The difficult graphical proofs of this paper just consist in
analyzing such graphs. This is done by splitting the punctured plane into simply-connected regions, and using the above proposition.
\end{remark}

%
\subsection{Generators of the TL Ideal}
\label{subsec-TLidealgens}
%

The sequence $\ii$ is assumed to be non-repeating.

\begin{prop}\label{prop-containment} 
  The TL ideal of $\emptyset$ contains $y_{i,j}\define
  f_if_j(f_i+2f_{i+1}+2f_{i+2}+\ldots+2f_{j-1}+f_j)$ over all $1\le i < j \le n$.

  The TL ideal of $\ii$ contains $z_{i,j,\ii} \define \frac{y_{i,j}}{g_ig_j}$
  where $g_i=f_i$ if $i \in \ii$, $g_i=1$ otherwise.
\end{prop}

We will prove that these actually generate the ideal in Proposition \ref{prop-generators}, but postpone the proof as it is long and unenlightening.

\begin{proof}

Adding 4 dots to $\alpha_i$, or 6 dots to a 6-valent vertex, we get

\begin{equation} \ig{.8}{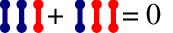}. \end{equation}

This is $y_{i,i+1}=f_if_{i+1}(f_i+f_{i+1})=(x_i-x_{i+1})(x_{i+1}-x_{i+2})(x_i-x_{i+2})$. Even though we are not allowing 6-valent vertices in our diagrams, we will sometimes express
$y_{i,i+1}$ as \begin{center} $\ig{.8}{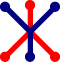}$ or $\ig{.8}{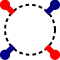}$ \end{center} to avoid having to consider sums of graphs (it's easier for me to draw!).

To obtain the other $y_{i,j}$, note the following equalties under the action of $S_{n+1}$ on $R$:

\begin{eqnarray}
s_if_{i+1} &=& f_i+f_{i+1}\\
s_{i+1}f_i &=& f_i+f_{i+1}\\
s_if_i &=& -f_i\\
s_if_j &=& f_j\ \ \mathrm{for}\ |i-j|>1
\end{eqnarray}

From this it follows by explicit calculation that

\begin{eqnarray}
s_{i-1}y_{i,j}-y_{i,j}=y_{i-1,j}\\
s_{j+1}y_{i,j}-y_{i,j}=y_{i,j+1}
\end{eqnarray}

Now, when we surround a polynomial $f$ with a $j$-colored circle and use (\ref{polyhole}), we are left with a $j$-colored double dot times $\partial_j(f)$, so we get $f-s_jf =
\partial_j(f)f_j$.

\plabeltwo{$f$}{360}{424}{$f-s_jf$}{414}{424}
\begin{equation} \label{circlearoundf} \ig{1.3}{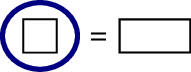} \end{equation}

Combining this with the calculations we just made, we see that a $j+1$ circle around $y_{i,j}$ will yield $y_{i,j+1}$ up to sign, etcetera. We now have numerous ways to express $\pm
y_{i,j}$: for any $i \le k \le j-1$ take $\alpha_k$ with 4 dots to get $y_{k,k+1}$, and then surround it with concentric circles whose colors, from inside to out, are $k+2,k+3,\ldots,j$
and then $k-1,k-2,\ldots,i$.

\igc{.8}{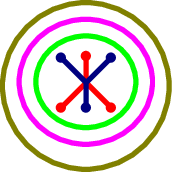}

Clearly the colors of the increasing sequence and those of the decreasing sequence are distant, so a sequence like $k-1,k+2,k+3,k-2,\ldots$ is also okay, or any permutation which
preserves the order of the increasing and the decreasing sequence individually.

For very similar reasons, $z_{i,j,\ii}$ is in the TL ideal of $\ii$. Adding two or three dots to (\ref{switch}), we get several more equations.

\begin{equation} \ig{.8}{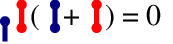} \end{equation}

\begin{equation} \ig{.8}{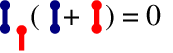} \end{equation}

\begin{equation} \ig{.8}{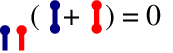} \end{equation}

Again, for various of these pictures we use shorthand like \begin{center} $\ig{.8}{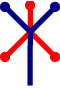}$ or $\ig{.8}{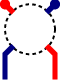}$ \end{center}

These give you $z_{i,i+1,\ii}$ in the case where at least one of $i,i+1 \in \ii$. Again by (\ref{polyhole}), putting a polynomial $f$ in the eye of a $j$-colored needle will yield
$\partial_j(f)=\frac{f-s_jf}{f_j}$ next to a $j$-colored boundary dot.

\plabeltwo{$f$}{360}{424}{$\partial_jf$}{414}{424}
\begin{equation} \label{needlearoundf} \ig{1.2}{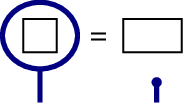} \end{equation}

This gives us several ways to draw $z_{i,j,\ii}$.

If neither $i$ nor $j$ are in $\ii$, then $z_{i,j,\ii}=y_{i,j}$ and is pictured as above, but with additional boundary dots put below to account for $\phi_{\ii}$. Since these extra dots
are generally irrelevant, we often do not bother to draw them.

If $i\in\ii$ and $j \notin \ii$, we have two ways of drawing $z_{i,j,\ii}$. One can take $\alpha_i$, connect one $i$ input to the outer boundary, add dots, and surround it with circles
colored $i+2,i+3,\ldots,j$.

\begin{equation} \ig{.8}{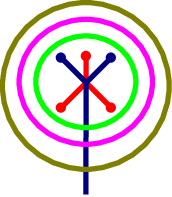} \label{case1} \end{equation}

Alternatively, take some $i<k<j$, add dots to $\alpha_k$, and surround it with circles forming an increasing sequence $k+2 \ldots j$ and a decreasing sequence $k-1 \ldots i$,
\emph{except} that the final $i$-colored circle is a needle.

\begin{equation} \ig{.8}{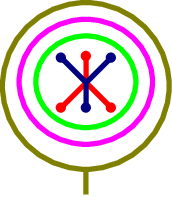} \label{case1a} \end{equation}

The case of $j \in \ii$ and $i \notin \ii$ is obvious.

If both $i,j \in \ii$ then we have several choices again. If $j=i+1$ then we must use 

\begin{equation} \ig{.8}{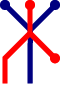} \label{case3} \end{equation}

but in general we may either repeat (\ref{case1}) with a $j$-needle instead of a $j$-circle

\begin{equation} \ig{.8}{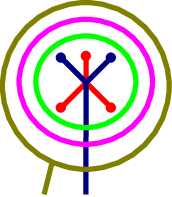} \label{case2} \end{equation}

or repeat (\ref{case1a}) with a $j$-needle instead of a $j$-circle.

\begin{equation} \ig{.8}{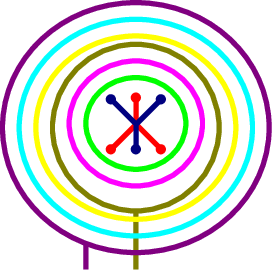} \label{case4} \end{equation}

In any case, it is clear that the polynomials above are in the TL ideal, and the claim is proven.
\end{proof}

Let us quickly consider the redundancy in this generating set of the ideal. When $i>j$ let $y_{i,j} \define y_{j,i}$ and $z_{i,j,\ii} \define z_{j,i,\ii}$.

\begin{cor} Suppose $\ii$ is non-empty, and fix an index $k \in \ii$. Then $I_{\ii}$ is
  generated by $z_{k,j,\ii}$ for $1\le j \le n$, $j \ne k$. None of these generators is redundant.

  None of the generators $y_{i,j}$ of $I_{\emptyset}$ are redundant.
\end{cor}

\begin{proof} We leave the checks of irredundancy to the reader, but a proof will also arise as a
  byproduct in the next section (see Remark \ref{rmk-irredundant}). 

  Suppose that $k \in \ii$ but $i,j \notin \ii$.  If $k<i<j$, then
  $z_{i,j,\ii}=y_{i,j}=f_iz_{k,j,\ii}-f_jz_{k,i,\ii}$ so that $z_{i,j,\ii}$ is redundant.  If
  $i<k<j$, then $z_{i,j,\ii}=f_iz_{k,j,\ii}+f_jz_{i,k,\ii}$. A similar statement holds for $i<j<k$.
  In the same vein, if $k,l \in \ii$ but $i \notin \ii$, then given $z_{k,l,\ii}$ only one of $z_{k,i,\ii}$
  or $z_{l,i,\ii}$ is needed, and if $k,l,m \in \ii$, then any 2 of the three pairwise relations
  will imply the third.
\end{proof}

%
\subsection{Graded Dimensions}
\label{subsec-gradeddim}
%

In this section, fix a non-repeating sequence $\ii$. We assume in this section that the generators of $I_{\ii}$ are precisely the polynomials described in Proposition
\ref{prop-containment}.

\begin{notation} An element of $R$ can be written as a polynomial in $f_i$, so let $x=f_1^{a_1}\ldots f_n^{a_n}$ be a general monomial. Choose any $\ii$, possibly empty. Given a monomial
$x$, let $J_x \subset \{1,\ldots,n\}$ be the subset containing $\ii$ and all indices $j$ such that $a_j \ne 0$. For a fixed subset $J$, let $R_J$ be the subset of all monomials $x$ with
$J_x=J$.  This inherently depends on the choice of $\ii$.

Under the map $R \to \HOM(U_{\emptyset},U_{\ii})$, the image of $R_J$ will be graphs where the colors appearing are precisely $J$. Every color in $\ii$ appears as a boundary dot, and
every $f_j$ corresponds to a double dot of that color. The case $J = \emptyset$ only occurs when $\ii = \emptyset$, and $R_\emptyset=\{1\}$. \end{notation}

To find a basis for $R/I_{\ii}$ we will use the Bergman Diamond Lemma \cite{Berg} for commutative rings:

\begin{defn} Let $A$ be a free commutative polynomial ring, where monomials are given a partial order with the DCC, compatible with multiplication in that $x<y \implies ax<ay$. Let $I$ be
an ideal generated by relations $r$ of the form $x_r = y_r$ where $x_r$ is a monomial and $y_r$ is a linear combination of monomials which are each less than $x_r$ in the partial order. A
\emph{reduction} is an application of a relation $r$ to replace $x_r$ with $y_r$, but not the other way around (a reduction always lowers the partial order on each term in a
polynomial). We say a polynomial $x$ \emph{reduces} to $y$ if $y$ can be obtained from $x$ by a series of reductions applied to monomials in $x$. A monomial is called \emph{irreducible}
if it does not have $x_r$ as a factor for any relation $r$. An \emph{inclusion ambiguity} is a monomial $x=ab$ where $x=x_r$ for some $r$, and $b=x_{r'}$ for some $r' \ne r$. An
\emph{overlap ambiguity} is a monomial $x=abc$ where $ab=x_r$ for some $r$ and $bc=x_{r'}$ for some $r' \ne r$. Each ambiguity has two natural reductions, and we say the ambiguity is
\emph{resolvable} if the two reductions are then jointly reducible to the same element.\end{defn}

\begin{lemma} (Bergman diamond lemma for commutative rings, \cite{Berg}) With these definitions in place, if every inclusion and overlap ambiguity is resolvable, then the images of the
irreducible monomials form a basis for $A/I$. \end{lemma}

This process may become more transparent from the example below; in addition, Bergman's paper has a number of nice examples for the trickier, non-commutative version. We treat two
separate cases, when $\ii=\emptyset$ and when $\ii \ne \emptyset$.

\begin{claim} Let $\ii=\emptyset$. We place the lexicographic order on monomials in $R$, so that $f_1 < f_2 < \ldots$. The relation $y_{i,j}=0$ for $i<j$ will be rewritten $f_if_j^2=
-f_if_j(f_i + 2\sum_{i<k<j}f_k)$, which replaces $f_if_j^2$ with a sum of monomials all lower in the order. For each $J \ne \emptyset$, the irreducible monomials in $R_J$ are
precisely $f_k^m \prod_{i \in J}f_i$, where $k$ is the minimal index in $J$ and $m \ge 0$ (note: the exponent of $f_k$ is $m+1 \ge 1$). When $J= \emptyset$, $1$ is irreducible.
Irreducibles form a basis for $R/I_{\emptyset}$. \end{claim}

\begin{proof} A monomial is irreducible if $f_if_j^2$ never appears as a factor for any $i<j$. Because of this, the classification of irreducible monomials in each $R_J$ is obvious. There
are no inclusion ambiguities between relations, since they are all homogeneous and degree 3. There are two kinds of overlap ambiguities, both labelled by a choice of $i<l<j$.

For the first ambiguity, one can reduce $x=f_if_lf_j^2$ by either reducing $f_lf_j^2$ or $f_if_j^2$. Applying the former reduction, $x \mapsto f_if_lf_j(-f_l - 2\sum_{l<k<j}f_k)$, which
has a term given by $-f_if_l^2f_j$ that can be further reduced, yielding $f_if_lf_j(f_i+2\sum_{i<k<l}f_k - 2\sum_{l<k<j}f_k)$. Applying the latter reduction, $x \mapsto f_if_lf_j(-f_i -
2\sum_{i<k<j}f_k) = f_if_lf_j(-f_i - 2\sum_{i<k<l}f_k - 2f_l - 2\sum_{l<k<j}f_k)$, which has a term given by $-2f_if_l^2f_j$ that can be further reduced, yielding $f_if_lf_j(-f_i -
2\sum_{i<k<l}f_k - 2\sum_{l<k<j}f_k + 2f_i + 4\sum_{i<k<l}f_k) = f_if_lf_j(f_i + 2\sum_{i<k<l}f_k - 2\sum_{l<k<j}f_k)$. Since these agree, the ambiguity is resolvable.

For the second ambiguity, one can reduce $x=f_if_l^2f_j^2$ be either reducing $f_if_l^2$ or $f_lf_j^2$. A very similar calculation shows that this ambiguity is resolvable as well.
Therefore the Bergman diamond lemma implies that irreducibles form a basis for the quotient. \end{proof}

\begin{remark} \label{rmk-irredundant} This also proves that none of the $y_{i,j}$ is redundant. Removing $y_{i,j}$ from the ideal, we may apply the same Bergman diamond lemma argument to
say that irreducibles form a basis for the quotient. However, with no $y_{i,j}$ the monomial $f_if_j^2$ is irreducible, and the quotient is larger than before. A similar statement can be
made about the $z_{k,j,\ii}$ below. \end{remark}

When $J \ne \emptyset$, the graded rank of the irreducibles in $R_J$ is $\frac{t^{2|J|}}{1-t^2}$. When $J$ is empty, the only element of $R_J$ is $1$. So the graded rank of
$R/I_{\emptyset}$ is $1 + \sum_{J\ne \emptyset}\frac{t^{2|J|}}{1-t^2}$. But $\sum_{J}t^{2|J|}=(1+t^2)^n$ since every $f_i$ may either appear or not appear, independently of every other.
Hence $\sum_{J\ne\emptyset}t^{2|J|}=(1+t^2)^n-1$. Putting it all together, the graded rank is $\frac{(1+t^2)^n-t^2}{1-t^2}=\frac{t^n\qtwo^n-t^2}{1-t^2}$.  Hence we have proven

\begin{claim} The graded dimension of $R/I_{\emptyset}$ is exactly $\epsilon_{\rm{cat}}(u_{\emptyset})$. \end{claim}

\begin{claim} Let $\ii \ne \emptyset$, and fix $k \in \ii$. We choose a different order on indices, where $k<k+1<k-1<k+2<k-2< \ldots$, and then place the lexicographic order on monomials.
The relation $z_{k,j,\ii}$ for $j \ne k$ will be rewritten in order-decreasing format as either $f_j^2=-f_j(f_k + 2\sum_{l}f_l)$ for $j \notin \ii$, or $f_j=-(f_k + 2\sum_{l}f_l)$ for
$j \in \ii$, where the sum is over $l$ between $k$ and $j$. Then the irreducible monomials in $R_J$ are precisely $f_k^m \prod_{j \in J \setminus \ii} f_j$ for $m \ge 0$. Irreducibles
form a basis for $R/I_{\ii}$. \end{claim}

\begin{proof} An irreducible polynomial will be a polynomial which does not have $f_j^2$ as a factor, for $k \ne j \notin \ii$, and does not have $f_j$ as a factor for $k \ne j \in \ii$.
The classification of irreducibles in $R_J$ is now obvious. There are no ambiguities whatsoever, so we are done by the Bergman diamond lemma. \end{proof}

The graded rank of irreducibles in $R_J$ is $\frac{t^{2|J|-2d}}{1-t^2}$, for $d$ the length of $\ii$ (remember that $\ii \subset J$). Thus the graded rank of $R/I_{\ii}$ is $\sum_{\ii
\subset J}\frac{t^{2|J|-2d}}{1-t^2} = \frac{(1+t^2)^{n-d}}{1-t^2}$, and the graded rank of $R/I_{\ii}\{d\}$ is $t^d\frac{(1+t^2)^{n-d}}{1-t^2}=\frac{t^n\qtwo^{n-d}}{1-t^2}$. Hence,

\begin{claim} The graded dimension of $R/I_{\ii}\{d(\ii)\}$ is exactly $\epsilon_{\rm{cat}}(u_{\ii})$. \end{claim}
	
This is clearly sufficient to prove Lemma \ref{mainlemma}, modulo Proposition \ref{prop-generators}.

%
\subsection{Weyl Lines and Disoriented Tubes}
\label{subsec-weyllines}
%

We now give two alternate interpretations of the TL ideals $I_{\ii}$. We continue to assume that $\ii$ is non-repeating and $z_{i,j,\ii}$ generates $I_{\ii}$.

\begin{defn} Let $V$ be the reflection representation of $S_{n+1}$, such that $R=\C[f_1,\ldots,f_n]$ is the coordinate ring of $V$. Note that the linear equations which cut out
reflection-fixed hyperplanes are precisely $w_{i,j}=f_i+f_{i+1}+\ldots+f_j =x_i-x_{j+1}$ for $i\le j$. A \emph{Weyl line} is a line in $V$ through the origin which is
defined by the intersection of reflection-fixed hyperplanes; it is given by a choice of $n-1$ transversely-intersecting reflection-fixed hyperplanes. Given a non-repeating sequence
$\ii$, we say a Weyl line is \emph{transverse to} $\ii$ if it is transverse to (i.e. not contained in) the hyperplanes $f_k=0$ for each $k \in \ii$. \end{defn}

\begin{prop}\label{prop-weyllines} The TL ideal of $\ii$ is the ideal associated with the union of all Weyl lines transverse to $\ii$ (with its reduced scheme structure). \end{prop}

\begin{example} Let $n=3$. One can check that $f_1f_2(f_1+f_2)=f_2f_3(f_2+f_3)=f_1f_3(f_1+2f_2+f_3)=0$ cuts out 7 lines in $V$, namely \begin{enumerate} \item $f_1 = f_2 = f_1+f_2 = 0$
\item $f_1 = f_3 = 0$ \item $f_2 = f_3 = f_2+f_3 = 0$ \item $f_1 = f_2+f_3 = f_1+f_2+f_3 = 0$ \item $f_1+f_2 = f_3 = f_1+f_2+f_3 = 0$ \item $f_2 = f_1+f_2+f_3 = 0$ \item $f_1+f_2 =
f_2+f_3 = 0$ \end{enumerate} These 7 lines are precisely the 7 lines cut out by the intersection of pairs of reflection-fixed hyperplanes. There are 6 reflection-fixed hyperplanes, given
by equations $f_1$, $f_2$, $f_3$, $f_1+f_2$, $f_2+f_3$, and $f_1+f_2+f_3$, or alternatively, by $x_i-x_j$ for $4\ge j > i \ge 1$. Intersecting pairs of hyperplanes will give a line, and
occassionally this line is forced to lie in a third hyperplane, as in the list above. One can check that this list covers all pairs of hyperplanes which give distinct lines as their
intersection. \end{example}

\begin{proof} This is not difficult to show, but since we have not seen it elsewhere, we provide a complete proof. First we show by induction on $n$ that the ideal $I_{\emptyset}$ cuts
out the Weyl lines with the reduced scheme structure. The case $n=1$ is trivial (and $n=2$ is also obvious).

For any $1 \le k \le n$, consider the hyperplane $f_k=0$ as an $n-1$-dimensional space $V'$, with an action of $S_{n+1}/<s_k> \cong S_n$. Giving $S_n$ a Coxeter structure with simple
reflections $s_i$ for $i\ne k$ (note that $s_{k+1}=(k+1, k+2) = (k, k+2)$ in the quotient), it is quite easy to see that $V'$ is the reflection representation of $S_n$. Moreover, the Weyl
hyperplanes are cut out by $w'_{i,j}=f_i+f_{i+1}+\ldots+f_j$ (where $f_k=0$ so it may be left out of the sum) for $i,j \ne k$, and the equivalent polynomials $y'_{i,j}$ also have the same
formulae, and are indexed by $i,j \ne k$. Therefore, for $i,j \ne k$, the images of $w_{i,j}$ are just $w'_{i,j}$, and the same for $y_{i,j}$ and $y'_{i,j}$. Moreover, if either $i$ or
$j$ equals $k$, then $y_{i,j}=0$ on $f_k=0$, and $w_{i,j}$ is redundant on $f_k=0$, being equal to some $w_{i',j'}$. By induction, $y'_{i,j}$ cut out the Weyl lines with the reduced
scheme structure on $V'$, and therefore the vanishing set of $y_{i,j}$ agrees with the Weyl lines on $f_k=0$.

If all $f_k \ne 0$, then it is easy to see that the $y_{i,j}$ cut out a single line with the reduced scheme structure, namely $-f_1=f_2=-f_3=\ldots=(-1)^nf_n$. This is a Weyl line, the
intersection of all $w_{i,i+1}$. We wish to show this is the only Weyl line transverse to all $f_k=0$. We can show this by induction as well (again, the base case $n=2$ is easy). Suppose
we are given $n-1$ transverse hyperplanes $w_{i,j}$. If any two both involve the index $n$, i.e. $w_{i,n}$ and $w_{j,n}$, then we may replace the pair with $w_{i,n}$ and $w_{i,j}$ since
they have the same intersection (and $w_{i,j}$ is not already in the set, or the intersection would not be transverse). So we may assume that at most one of the chosen hyperplanes
involves the index $n$. But then we have $n-2$ transverse hyperplanes which only involve indices $\{1,\ldots,n-1\}$, which must then be mutually transverse to $f_n=0$. Letting $V'$ be the
hyperplane $f_n=0$ viewed as a reflection representation as above, we have $n-2$ transverse hyperplanes which cut out a Weyl line transverse to $f_k=0$ for all $1 \le k \le n-1$. By
induction, that Weyl line is $-f_1=f_2=-f_3=\ldots=(-1)^{n-1}f_{n-1}$ (which holds true modulo $f_n=0$). But repeating the same argument for the index $k$ instead, we leave out the $k$-th
term and get $-f_1=f_2=\ldots=\hat{(-1)^kf_k}=\ldots=(-1)^nf_n$ modulo $f_k=0$. Together, all these equalities imply that $-f_1=f_2=\ldots=(-1)^nf_n$ everywhere.

One might be worried, because of the restrictions used in the induction step, that $I_{\ii}$ does not give the reduced structure on the Weyl lines at the origin. However, $I_{\ii}$ is a
homogeneous ideal which cuts out a reduced 0-dimensional subscheme of $\mathbb{P}(V)$, so that its vanishing on $V$ is the cone of a reduced scheme, and hence is reduced. This
concludes the proof that $I_{\emptyset}$ cuts out the Weyl lines with the reduced scheme structure.

For $\ii \ne \emptyset$, $I_{\emptyset} \subset I_{\ii}$ and the vanishing of $I_{\ii}$ is contained in that of $I_{\emptyset}$. Choose $k \in \ii$. If $f_k=0$ then $z_{k,k+1,\ii}$ is
equal to $f_{k+1}^a$ where $a=1,2$ depending on whether $k+1 \in \ii$, but either way we get that $f_{k+1}=0$. Then $z_{k,k+2} = f_{k+2}^a$ for $a=1,2$, and so forth. Therefore $f_k=0$
only intersects the vanishing of $I_{\ii}$ at the origin (as sets). It is clear that, on the open set where $f_k \ne 0$ for all $k \in \ii$, the polynomials $z_{i,j,\ii}$ and $y_{i,j}$
have the same vanishing (as schemes), since they differ by a unit. The same cone argument shows that $I_{\ii}$ gives the reduced structure at the origin.\end{proof}

\begin{remark} In particular, $I_{\emptyset}$ is contained in every ideal, and the category $\mTLC_1$ is manifestly $R/I_{\emptyset}$-linear. \end{remark}

\begin{remark} Let $Z$ be the union of all Weyl lines in $V$. The previous results should lead one to guess that the Temperley-Lieb algebra should be connected to the geometry of the
$S_{n+1}$ action on $Z$ via $\mTLC$, in much the same way that the Hecke algebra is connected to the reflection representation via $\mHC$ (see \cite{Soe4}). However, at the moment we
have no way to formulate the category $\mTLC$ in terms of coherent sheaves on $Z \times Z$ (i.e. $R/I_{\emptyset}$-bimodules) or the derived category thereof. Describing $\mTLC$ using sheaves on
$Z$ seems like an interesting question.

As an example of the difficulties, let $U_i$ be the bimodule $R/I_i \otimes R/I_i \{-1\}$, where the tensor is over $R^{s_i}$; this should be the equivalent of the Soergel bimodule $B_i$.
Then there is a degree 1 map $R/I_{\emptyset} \to U_i$ sending $1$ to $x_i \ot 1 - 1 \ot x_{i+1}$ (the boundary dot on the top), but there is no degree 1 map $U_i \to R/I_{\emptyset}$
(the boundary dot on the bottom); such a map should send $1 \ot 1$ to $1$. There is only a degree 3 map, sending $1 \ot 1$ to $f_i$ (the boundary dot with a double dot). A similar problem
occurs again: the trivalent vertex seems to be defined only in one direction. \end{remark}

Now we describe briefly the topological intuition associated with the category $\mTLC$, and another way to view $I_{\ii}$. These remarks will not be used in the remainder of the paper,
nor will we give a proof. The reader should be acquainted with the section on $\mathfrak{sl}_2$-foams in Vaz's paper \cite{Vaz}.

\begin{remark} Let $\mc{F}$ be the functor from $\mc{TLC}_1$ to the category of disoriented cobordisms $\Foam_2$, as defined in Vaz's paper. If $f_i$ is the double dot colored $i$, then
one can easily see that $\mc{F}$ sends $f_i$ to a tube connecting the $i$th sheet to the $(i+1)$th sheet, with a disorientation on it. If the double dot appears in a larger morphism
$\phi$, such that in $\mc{F}(\phi)$ the $i$th sheet and the $(i+1)$th sheet are already connected by a saddle or tube, then adding another tube between them does nothing more than add a
disoriented handle to the existing surface. Note that the map $\phi_{\ii}$ previously defined will connect the $i$th sheet to the $(i+1)$th sheet for any $i \in \ii$.

Suppose that the $i$th, $(i+1)$th, and $(i+2)$th sheets are all connected in a cobordism. Then $f_i$ adds a handle on the left side of the $(i+1)$th sheet, $f_{i+1}$ adds a handle on the
right side, and these two disoriented surfaces are equal up to a minus sign in $\Foam_2$. This fact is essentially the statement that:

\igc{.8}{2dotswitch}

In other words, the algebra $\Bbbk[f_1,\ldots,f_n]$ maps to $\Foam_2$, sending $f_i$ to the disoriented tube between the $i$th and $(i+1)$th sheet. The ideal $I_{\ii}$ is clearly in the
kernel of this action when applied to the cobordism $\mc{F}(\phi_{\ii})$. In fact, it is precisely the kernel, using the argument of Proposition 4.2 in \cite{MV}: for any distinct
monomials in a basis for $R/I_{\ii}$, their image in $\Foam_2$ will have independent evaluations with respect to some closure of the cobordism. We do not do the calculation here. The
usual arguments involving adjoint pairings imply that the faithfulness of the functor $\mc{F}$ can be checked on $\Hom(\emptyset,\ii)$. Therefore the functor $\mc{F}$ is faithful.
\end{remark}

%
\subsection{Proof of Generation}
\label{subsec-pfprop}
%

\begin{prop}\label{prop-generators}
  The TL ideal $I_{\emptyset}$ is generated by $y_{i,j}\define
  f_if_j(f_i+2f_{i+1}+2f_{i+2}+\ldots+2f_{j-1}+f_j)$ over all $1\le i < j \le n$.

  The TL ideal $I_{\ii}$ is generated by all $z_{i,j,\ii} \define \frac{y_{i,j}}{g_ig_j}$
  where $g_i=f_i$ if $i \in \ii$, $g_i=1$ otherwise.
\end{prop}

We wish to determine the ideal generated by $\alpha_k$ inside $\HOM(\emptyset,\ii)$, for $\ii$ non-repeating. As discussed in Remark \ref{puncturedremark} (where $\alpha_k$ is defined),
our goal is to take any graph $\Gamma$ on the punctured plane, with $\ii$ as its outer boundary and $k(k+1)k(k+1)$ as its inner boundary, plug $\alpha_k$ into the puncture, and reduce it
to something in the ideal generated by the pictures of Section \ref{subsec-TLidealgens}.

\plabeltwo{$\alpha_k$}{692}{248}{$\Gamma$}{673}{258}
\igc{1.2}{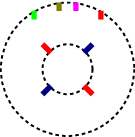}

Our coloring conventions for this chapter will be that blue always represents the index $k$, red represents $k+1$, and other colors tend to be arbitrary (often, the
number of other colors appearing is also arbitrary). However, it will often happen that colors will appear in increasing or decreasing sequences, and these will be
annotated as such.  Note that blue or red may appear in the outer boundary as well, but at most once each.

Let us study $\Gamma$, and not bother to plug in $\alpha_k$. The only properties of $\alpha_k$ which we need are the following:

	\begin{equation} \label{alphaprop1} \ig{.8}{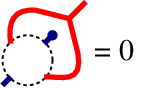} \end{equation}
	This follows from (\ref{dot6surrounded}), or just from isotopy. The same holds with colors switched.
	
	\begin{equation} \label{alphaprop2} \ig{.8}{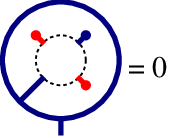} \end{equation}
	This is because the diagram reduces to a $k$-colored needle, with $f=f_{k+1}(f_k+f_{k+1})$ inside.  But $f$ is
	fixed by $s_k$, so it slides out of the needle, and the empty needle is equal to 0. A similar
	equality holds with colors switched.
	
	 \begin{equation} \label{alphaprop3} \ig{.8}{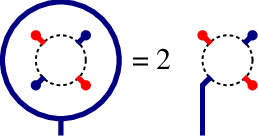} \end{equation}
	This follows from the above and the dot forcing rules.
	
	The final property we use is that any graph only using colors $<k-1$ or $>k+2$ can slide freely across the puncture.

Note however that, say, an arbitrary $k-3$ edge can not automatically slide across the puncture, because a $k-2$ edge might be in the way, and this could be in turn
obstructed by a $k-1$ edge, which can not slide across the blue at all.

The one-color reduction results apply to any simply-connected planar region, so we may assume (without even using the relation (\ref{switch})) that in a simply
connected region of our choice, the $i$-graph for each $i$ is a simple forest with double dots. Any connected component of an $i$-graph that does not encircle the
puncture will be contained in a simply-connected region, and hence can be simplified; this will be the crux of the proof. The proof is simple, but has many cases.

\begin{remark} We will still need to use relation (\ref{switch}) as we simplify graphs. \end{remark}

We will treat cases based on the ``connectivity" of $\Gamma$, that is, how many of the blue and red boundary lines in the inner and outer boundary are connected with each
other. We will rarely perform an operation which makes the graph more connected. At each stage, we will reduce the graph to something known to be in the ideal, or break
edges to decrease the connectivity. We call an edge coming from the puncture an \emph{interior line} and one coming from the outer boundary an \emph{exterior line}.

Note also that any double dots that we can move to the exterior of the diagram become irrelevant, since the picture with those double dots is in the ideal generated by
the picture without double dots. Also, any exterior boundary dots are irrelevant, since they are merely part of the map $\phi_{\ii}$ and do not interfere with the rest
of the diagram at all.

\emph{Step 1:} Suppose that the two interior red lines are in the same component of $\Gamma_{k+1}$. Then there is some innermost red path from one to the other, such
that the interior of this path (the region towards the puncture) is simply-connected. Applying reductions, we may assume that the $k$-graph in this region consists of
a blue boundary dot with double dots, and the $k+1$-graph and $k+2$-graph each consist only of double dots. We may assume all double dots occur right next to one of the red
lines coming from the puncture. The current picture is exactly like that in (\ref{alphaprop1}), except that there may be double dots inside, and other colors may be
present (also, there could be more red spokes emanating from the red arc, but these can be ignored or eliminated using (\ref{assoc1}) and tree reduction). However, the double dots may
be forced out of the red enclosure at the cost of potentially breaking the red edge, and breaking it will cause the two red interior lines to be no longer in the same
connected component. If there are no double dots, then all the remaining colors (which are $<k-2$ or $>k+1$) may be slid across the red line and out of the picture.
Hence we are left with the exact picture of (\ref{alphaprop1}), which is zero.

Thus we may assume that the two red lines coming from the puncture are not in the same component. The same holds for the blue lines.

\emph{Step 2:} Suppose that the component of one of the interior blue lines wraps the puncture, creating an internal region (which contains the puncture). Again, reducing in that
internal region, the other interior blue line can not connect to the boundary so it must reduce to a boundary dot (with double dots), the reds may not connect to each other so each
reduces to a boundary dot, and as before we are left in the picture of (\ref{alphaprop2}) except possibly with double dots and other colors. If there are no double dots, all other colors
may be slid out, and the picture is zero by (\ref{alphaprop2}). Again, we can put the double dots near the exterior, and forcing them out will break the blue arc. It is still possible
that some other cycle still allows that component to wrap the puncture; however, this process need only be iterated a finite number of times, and finitely many arcs broken, until that
component no longer wraps the puncture.

So we may assume that the component of any interior line, red or blue, does not wrap the puncture. That component is contained in a simply-connected region, so it
reduces to a simple tree. Hence, we may assume that the components of interior lines either end immediately in boundary dots, or connect directly to an external line of
the same color (at most one such exists of each color).

\emph{Step 3:} Suppose that there is a blue edge connecting an internal line directly to an external one. Consider the region $\Gamma'$:

\plabel{$\Gamma'$}{714}{252}
\igc{.8}{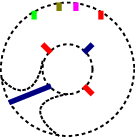}

Then $\Gamma'$ is simply-connected. Other colors in $\Gamma$ may leave $\Gamma'$ to cross through the blue line; however, the colors $k-1,k,k+1$ may not. Therefore,
reducing within $\Gamma'$, we may end the internal blue line in a boundary dot and eliminate all other instances of the color blue (since they become irrelevant double dots on the
exterior), reduce red to a simple forest where the two interior lines are not connected (again, ignoring irrelevant double dots), and reduce $k-1$ to either the empty diagram or an
external boundary dot (depending on whether $k+1 \in \ii$). Once this has been accomplished, the absence of the color $k-1$ implies that we may slide $k-2$ freely across the
puncture! The color $k-2$ can be dealt with in the entire disk, which is simply connected, so it reduces to the empty diagram or an external boundary dot (depending on whether $k+2
\in \ii$), with extraneous double dots. Then we may deal with color $k-3$, and so forth.

Thus, the existence of the blue edge implies that all colors $<k$ can be ignored: they appear in irrelevant double dots, in irrelevant boundary dots, or not at all. Similarly, the
existence of a red edge allows us to ignore all colors $>k+1$.

\emph{Step 4:} Let us only consider components of graphs which do not meet the internal boundary.

\begin{lemma} Consider a component of a graph on a punctured disk, which does not meet the internal boundary, and which meets the external boundary at most once. Then
it can be reduced to one of the following, with double dots on the exterior: the empty graph; a boundary dot; a circle around the puncture; a needle coming from the
external boundary, with its eye around the puncture. \end{lemma}

\begin{proof} Suppose that the component splits the punctured plane into $m$ regions. If the component is contained in a simply-connected part of the punctured plane, we are done. This
is always true for $m=1$. So we may suppose that $m\ge 2$ and we have two distinguished regions: the external region, and the region containing the puncture. Any other region is one of
two kinds, as illustrated in the following equality (due to (\ref{assoc1}):

\igc{.8}{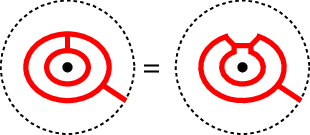}

On the right side we have a region which is contained in a simply-connected part, and thus can be eliminated by reduction (see Proposition \ref{prop-reduction}). On the left side the
region is not contained in a simply-connected part, nor does it contain the puncture. However, any such region can be altered, using (\ref{assoc1}) as in the heuristic example above,
into a cycle of the first kind. Therefore, we may assume there are exactly 2 regions.

In the event that there are two regions, we have a cycle which surrounds around the puncture, and may have numerous branches into both regions, internal and external. However, each
branch must be a tree lest another region be created. These trees reduce in the usual fashion, and therefore the internal branches disappear, and the external branches either disappear
or connect directly to the single exterior boundary. Thus we have either a needle or a circle. Double dots, as usual, can be forced out of the way possibly at the cost of breaking the
cycle, and reducing to the case $m=1$. \end{proof}

Let us now examine the remaining cases. We shall ignore all parts of a graph which are double dots on the exterior, or are external boundary dots.

\emph{Case 1:} Both a blue edge and a red edge connect an internal line to an external line. Then, as in Step 3, all other colors can be ignored, and the entire
graph is

\igc{.8}{alphasomedots}

This, as explained in Section \ref{subsec-TLidealgens}, is $z_{k,k+1,\ii}$.

\emph{Case 2:} A blue edge connects an internal line to an external line, and both red internal lines end in boundary dots. As discussed in Step 3, we may ignore all colors
$<k$, and both colors $k$ and $k+1$ do not appear in a relevant fashion outside of what is already described. We may ignore the presence of any double dots. However,
there may be numerous circles and needles colored $\ge k+2$ which surround the puncture and cross through the blue line, in an arbitrary order.

\igc{.8}{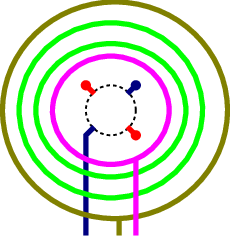}

\begin{claim} The sequence of circles and needles can be assumed to form an increasing sequence of colors, from $k+2, k+3, \ldots$ until the final color, and that only
the final color may be a needle. \end{claim}

\begin{proof} If the innermost circle/needle is not colored $k+2$, then it may slide through the puncture, and will evaluate to zero by (\ref{needle}). So suppose the innermost is $k+2$.
If it is a needle, not a circle, then there can be no more $k+2$-colored circles, and no $k+3$-colored circles. Color $k+4$ can be pulled through the middle so resolved on the entire
disk, and hence can be ignored, and so too with $k+5$ and higher. This is the ``needle" analogy to the conclusion of Step 3: the existence of a $m$-colored needle around the puncture and
the lack of $m$ or $m+1$ on the interior of the needle will allow us to ignore all colors $ \ge m+1$.

So suppose it is a $k+2$-colored circle. If the next circle/needle is colored $\ge k+4$ then it slides through the $k+2$ circle and the puncture, and evaluates to zero. If
the next circle/needle is \emph{also} colored $k+2$, then we may use the following calculation to ignore it. The calculation begins by using (\ref{iidecomp}).

\begin{equation} \ig{.8}{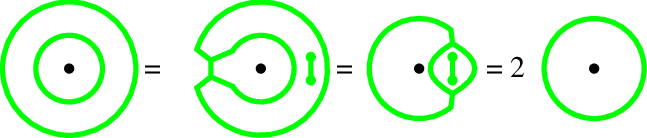} \label{twolinesannulus} \end{equation}

Thus we may assume that the next circle/needle is colored $k+3$. Again, if it is a needle, then we can ignore all other colors, and our picture is complete.

Similarly, the next circle/needle can not be colored $\ge k+5$ lest it slide through, and it can not be colored $k+3$ lest we use (\ref{twolinesannulus}). If it is
colored $k+2$, then we may use the following calculation to ignore it. The calculation begins by using (\ref{switch}), and assumes green and purple are adjacent.

\begin{equation} \ig{.8}{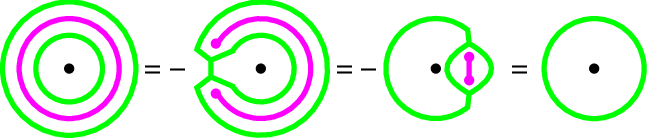} \label{ipiannulus} \end{equation}

Thus we can assume the next circle/needle is colored $k+4$. If it is a needle, then all colors $k+5$ and higher can be ignored. Additional circles of color $k+2$ could
run through the needle, but these could be slid inwards and reduced as before. So if it is a needle, our picture is complete.

Finally, the next circle/needle can not be colored $\ge k+6$ lest it slide, $k+4$ lest we use (\ref{twolinesannulus}), $k+3$ lest we use (\ref{ipiannulus}), or $k+2$
lest we slide it inside and reduce it as above. Hence it is colored $k+5$, and if it is a needle, we are done. This argument can now be repeated ad infinitum.
\end{proof}

Thus our final picture yields $z_{k,j,\ii}$ as in (\ref{case1}) or (\ref{case2}).

Note that the case of a red edge works the same way, with a decreasing sequence instead of an increasing sequence.

\emph{Case 3:} All the internal lines end in boundary dots. We may assume that the remainder of the graph consists in circles/needles around this diagram, but have no
restrictions at the moment on which colors may appear.

\begin{claim} We may assume that the colors in circles/needles form an increasing sequence from $k+2$ up, and a decreasing sequence from $k-1$ down (these sequences do
not interact, so w.l.o.g. we may assume the increasing sequence comes first, then the decreasing one). Only the highest and lowest color may be a needle. \end{claim}

\begin{proof} The method of proof will be the same as the arguments of the previous case.

Consider the innermost circle/needle. If it is colored $k$ or $k+1$, then we may use (\ref{alphaprop3}) to reduce the situation to a previous case. If it is colored
$\ge k+3$ or $\le k-2$ then it slides through the puncture. So we may assume it is $k+2$ or $k-1$. If it is a $k+2$-colored (resp. $k-1$-colored) needle, then the usual
arguments imply that all colors $>k+2$ (resp. $<k-1$) can be ignored. This same argument with needles will always work, so we will not discuss the circle/needle
question again, and speak as though everything is a circle.

Assume that the first colors appearing are an increasing sequence from $k+2$ to $i$ and then a decreasing sequence from $k-1$ to $j$. Note that either sequence may be
empty. If the next color appearing is $\le j-2$ then it slides through the whole diagram and the puncture, and evaluates to zero. If the decreasing sequence is
non-empty and the next color is $j$ then we use (\ref{twolinesannulus}); if it is $\ge j+1$ and $\le k-1$ then we slide it as far in as it will go and use
(\ref{ipiannulus}). If the decreasing sequence is non-empty and the next color is $k$ then one can push it almost to the center, and use the following variant of
(\ref{ipiannulus}):

\begin{equation} \ig{.8}{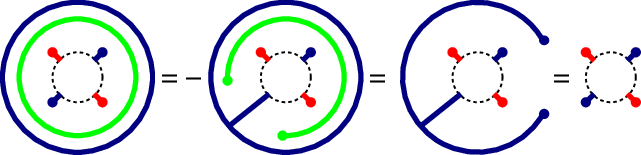} \label{ipiannulus2} \end{equation}

In this picture, green is $k-1$, and is the only thing in the way of the blue circle. The first equality uses (\ref{switch}), and the second equality uses (\ref{dotslidenear}), and
eliminates the terms which vanish due to (\ref{alphaprop2}).

Continuing, if the decreasing sequence is empty and the next color is $k$ then we may use (\ref{alphaprop3}) as above. Any colors which are $\ge k+1$ do not depend on
the increasing sequence, and instead use the exact analogs for the increasing sequence.

Hence, in any case in which the next color appearing is not $i+1$, or $j-1$, or the beginning of a new increasing/decreasing series, we may simplify the diagram to
ignore the new circle. Induction will now finish the proof.
\end{proof}

Therefore, the resulting diagram is equal to $z_{i,j,\ii}$, matching up either with (\ref{case1a}) or (\ref{case4}).

Since every possible graph can be reduced to a form which is demonstrably in the ideal generated by $z_{i,j,\ii}$, we have proven that these elements do in fact
generate the TL ideal $I_{\ii}$.

\section{Irreducible Representations}
\label{sec-repns}

In this section, we may vary the number of strands appearing in the Temperley-Lieb algebra. When $\mTL$ appears it designates the Temperley-Lieb algebra on $n+1$ strands, but
$\mTL_k$ designates the algebra on $k$ strands.

%
\subsection{Cell Modules}
\label{subsec-cellmodules}
%

The Temperley-Lieb algebra has the structure of a cellular algebra, a concept first defined by Graham and Lehrer \cite{GL}. One feature of cellular algebras is that they are equipped with
certain modules known as \emph{cell modules}. Cell modules provide a complete set of non-isomorphic irreducible modules in many cases (such as $\mTL$ in type $A$). Cell modules come
equipped with a basis and a bilinear form, making them obvious candidates for categorification. We will not go into detail on cellular algebras here, or even use their general properties;
instead we will describe the cell modules explicitly and pictorially for the case of $\mTL$, where things are unusually simple. Nothing in this section or the next is particularly
original, and we state some standard results without proof.

\begin{notation} Consider a crossingless matching in the planar strip between $n$ points on the bottom boundary
and $m$ points on the top. We call this briefly an $(n,m)$ \emph{diagram}. In the terminology of \cite{FG},
there are two kinds of arcs in a diagram: \emph{horizontal arcs}, which connect two points on the top (let us
call it a \emph{top arc}), or two points on the bottom (\emph{bottom arc}); and \emph{vertical arcs}, which
connect a point on the top to one on the bottom. Elsewhere in the literature, vertical arcs are called
\emph{through-strands}. An $(n,k)$ diagram with exactly $k$ through-strands (and therefore no top arcs) has an
isotopy representative with only ``caps" (local maxima) and no ``cups" (local minima) so it is called an
$(n,k)$ \emph{cap diagram}. A $(k,n)$ diagram with $k$ through-strands is called a $(k,n)$ \emph{cup diagram}.
\end{notation}

The set of all $(n,m)$ diagrams can be partitioned by the number of through-strands. Any $(n,m)$ diagram with
$k$ through-strands can be expressed as the concatenation of a $(n,k)$ cap diagram with a $(k,m)$ cup diagram
in a unique way. For an illustration of this concept, see Figure \ref{tlfigure}.

In an $(m,m)$ diagram the number $l$ of top arcs equals the number of bottom arcs, and if $k$ is the number of through-strands then $k+2l=m$. We will typically use $k$ and $l$ to
represent the number of through-strands and top arcs in an $(m,m)$ diagram henceforth.

\begin{notation} Let $X$ be the set of all $(n+1,n+1)$ diagrams. Let $\omega$ be the endomorphism of $X$ sending each diagram to its vertical flip. We will write the operation on diagrams of
reduced vertical concatenation by $\circ$: $a \circ b$ places $a$ above $b$, and removes any circles. Let $X_k$ be the set of crossingless matchings with exactly $k$ through-strands.
Let $M_k$ be the set of all $(n+1,k)$ cap diagrams, so that $\omega(M_k)$ is the set of all $(k,n+1)$ cup diagrams. \end{notation}

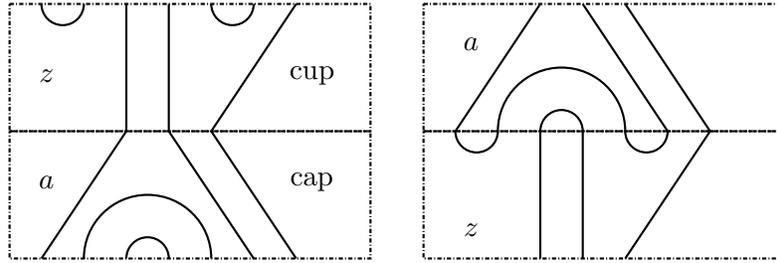
\begin{figure}
\label{tlfigure}
\labellist
\small\hair 2pt
\pinlabel $a$ at 230 413
\pinlabel $z$ at 230 453
\pinlabel $a$ at 390 465
\pinlabel $z$ at 390 395
\pinlabel cap at 330 413
\pinlabel cup at 330 453
\endlabellist
\igc{1}{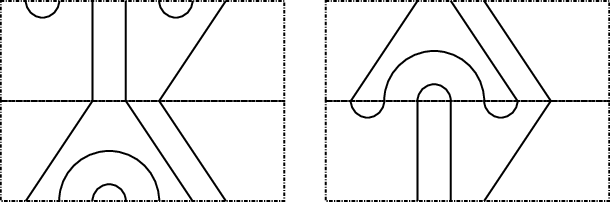}
\caption{On the left side, a $(7,7)$ diagram with $k=3$ through-strands and $l=2$ top arcs (resp. bottom arcs) is decomposed into a $(7,3)$ cap diagram $a$ composed with a $(3,7)$ cup diagram $z$. On the right side, an element of $\mTL_3$ is obtained by composing $a$ and $z$ in the opposite order.}
\end{figure}

\begin{defn} Let $L_k$ be the free $\Ztt$-module spanned by $M_k$, the $(n+1,k)$ cap diagrams. We place a right $\mTL$-module structure on $L_k$ by concatenation, where circles become factors of $\qtwo$ as usual,
and any resulting diagram with fewer than $k$ through-strands is sent to $0$. This is the \emph{cell module} for cell $k$, and it is irreducible. \end{defn}

\begin{example} The only diagram in $X_{n+1}$ corresponds to the identity map in $\mTL$. The cell module $L_{n+1}$ has rank 1 over $\Ztt$, and its generator is killed by all $u_i$. We
will take this as the definition of the \emph{sign representation} of $\mTL$. \end{example}

\begin{example} The next cell module $L_{n-1}$ has rank $n$ over $\Ztt$, having generators $v_i$, $i=1 \ldots n$ (see Figure \ref{cellmodex}), such that \begin{eqnarray} v_ju_i & = & \lbrace \begin{array}{ccc} \qtwo
v_j & \textrm{ if} & i=j\\ v_i & \textrm{ if} & i \textrm{ and }  j \textrm{ are adjacent} \\ 0 &\textrm{ if} & i  \textrm{ and } j  \textrm{ are distant} \end{array}.  \end{eqnarray} 
\end{example}

\begin{figure}
	
	$\ig{.5}{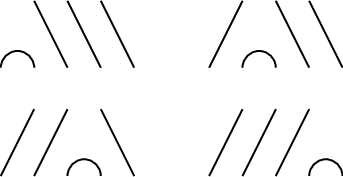}$
	\caption{A basis for the cell module $L_{n-1}$, consisting of $(n+1,n-1)$ cap diagrams (here, $n=4$).}
	\label{cellmodex}
\end{figure}

Given a $(n+1,k)$ cap diagram $a$ and a $(k,n+1)$ cup diagram $z$, there are two things we can do: take the composition $z \circ a$ to obtain an element called $c_{z,a}$ of $X_k$; or take
the composition $a \circ z$ to get an element of $\mTL_k$ (there may be additional circles created, and the final diagram may have fewer than $k$ through-strands). Both compositions have
the same closure on the punctured plane. Note that $\omega(c_{z,a})=c_{\omega(a),\omega(z)}$. The seemingly extraneous use of the notation $c_{.,.}$ is standard for cellular algebras.

\begin{prop} There is, up to rescaling, a unique pairing $(,) \colon L_k \times L_k \to \Zttt$ for which $u_i$ is self-adjoint, that is $(au_i,b)=(a,bu_i)$. Given cup diagrams $a$ and $b$
in $M_k$ we evaluate $(a,b)$ by considering the closure of $c_{\omega(a),b} \in \mTL$, or equivalently the closure of $b \circ \omega(a) \in \mTL_k$. If the diagram has nesting number $k$
we return a scalar times $\qtwo$ raised to the number of circles; if it has nesting number $<k$ we return zero. This is precisely the evaluation $\epsilon(c_{\omega(a),b})$ for some
well-defined trace on $\mTL$ supported on nesting number $k$ (which are unique up to rescaling). \end{prop}

%
\subsection{Some Induced Sign Representations}
\label{subsec-quotients}
%

Cell modules are naturally subquotients of the cellular algebra itself, viewed as a free module (see \cite{GL}). For our purposes, we will describe the cell modules as subquotients of
$\mTL$ in a different way, which will be more convenient to categorify diagrammatically. Taking the inclusion $\mTL_J \to \mTL$ for some sub-Dynkin diagram $J$, we can induce the sign
representation of $\mTL_J$ up to $\mTL$. This is the quotient of $\mTL$ by the right ideal generated by $u_i, i \in J$. In a future paper we will describe, for both the Hecke and
Temperley-Lieb algebras, a diagrammatic way to categorify the induction of both the ``sign" and ``trivial" representations of sub-Dynkin diagrams, but for this paper we restrict to a
specific case. For the sub-Dynkin diagram which contains every index except $i$, let $I_i$ be the corresponding ideal (generated by $u_j$ for $j \ne i$), and consider the induced sign
representation $V^i=\mTL/I_i$. Let $l_i=\min(i,n+1-i)$ and let $k_i=n+1-2l_i$. It turns out that we can embed $L_{k_i}$ inside $V^i$, as shown explicitly below, and we shall categorify
both modules accordingly. For this reason, we use $L^i$ to denote $L_{k_i}$. Note that every possible $L_k$ can be achieved as some $L^i$ with the exception of $L_{n+1}$.

For the rest of this section, fix an index $i \in I$. We define a module $V^i$ over $\mTL$ abstractly, and then prove that this module is isomorphic to $\mTL/I_i$. 

\begin{defn} For $0 \le l \le l_i$ (and letting $k=n+1-2l$ as always), let $a_k^i$ be the following $(k,n+1)$ cup diagram with $l$ top arcs, where the innermost top arc always connects
$i$ to $i+1$:

\labellist
\small\hair 2pt
\pinlabel $l=0$ at 255 381
\pinlabel $l=1$ at 367 381
\pinlabel $l=2$ at 255 350
\pinlabel $l=3$ at 367 350
\endlabellist
\igc{.8}{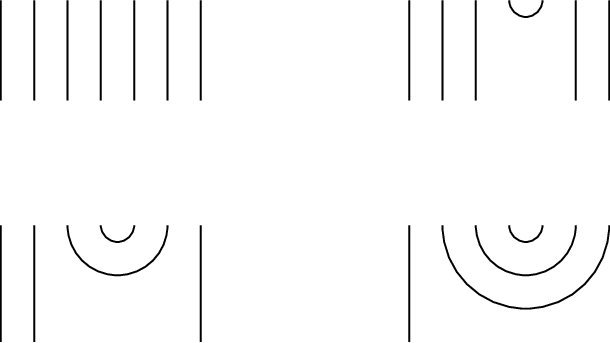}

Let $X_k^i \subset X_k$ consist of all matchings of the form $c_{a_k^i,b}$ for $b\in M_k$. Let $X^i$ be the disjoint union of all $X_k^i$ for $0 \le l \le l_i$, and let $V^i$ be the free
$\Ztt$-module with basis $X^i$. There is a distinguished element $\1$ of this basis, the unique member of $X_{n+1}^i$. Let $\mTL$ act on $V^i$ on the right by viewing elements of $V^i$ as
though they were in $\mTL$, using the standard multiplication rules, and then killing any terms whose diagrams are not in $X^i$. \end{defn}

The elements of $X^i$ exhaust those elements of $X$ where the only \emph{simple} top arcs (those connecting $j$ to $j+1$ for some $j$) connect $i$ to $i+1$. Any crossingless matching
with a simple top arc connecting $j$ to $j+1$ has an expression in $\mTL$ as a monomial $u_{\ii}$ which begins with $u_j$. The converse is also true. Thus $X^i$ are the elements of $X$
for which \emph{every} expression of the matching begins with $u_i$. This motivates the definition.

While something does need to be checked to ensure that this defines a module action, it is entirely straightforward. In the Temperley-Lieb algebra, things are generally easy to prove
because products of monomials always reduce to another monomial (with a scalar), not a linear combination of multiple monomials. Therefore, checking the associativity condition for being
a module, say, involves showing that both sides of an equation are the same diagram in $X^i$, or that both sides are 0. This module is cyclic, generated by $\1$, and $I_i$ is clearly in
the annihilator of $\1$, so that $\mTL/I_i$ surjects onto $V^i$. One could prove the following by bounding dimensions.

\begin{claim} The modules $V^i$ and $\mTL/I_i$ are isomorphic. \end{claim}

There is a (cellular) filtration on $V^i$, given by the span of $X_{\le k}^i$, diagrams with at most $k$ through-strands (call it $V_{\le k}^i$). Clearly, each subquotient in this
filtration has a basis given by $X^i_k$, or in other words by the elements $c_{a_k^i,b}$ for $b \in M_k$. It is an easy exercise that this subquotient is isomorphic to the cell module
$L_k$, under the map sending $b \in M_k$ to $c_{a_k^i,b}$. There is one subquotient for each $0 \le l \le l_i$.

\begin{claim} The module $L^i$ is a submodule of $V^i$. \end{claim}

\begin{proof} Letting $l=l_i$ and $k=k_i$, the final term in the filtration is precisely $L^i \cong V_{k_i}^i$. \end{proof}

Having explicitly defined the embedding $L^i \subset V^i$, we pause to investigate adjoint pairings on $V^i$.

\begin{prop} Consider the $\Zttt$-module of semi-linear pairings on $V^i$ where $(xu_j,y)=(x,yu_j)$ for all $j$. Consider the $l_i+1$ functionals on this space, which send a pairing to
$(\1, c_{a^i_k,\omega(a^i_k)})$ for various $k=n+1-2l$, $0 \le l \le l_i$. Then these linear functionals are independent and yield an isomorphism between the space of pairings and a free
module of rank $l_i+1$. \end{prop}

Note that, using adjunction, one can check that $\qtwo^l(\1, c_{a^i_k,\omega(a^i_k)})=(c_{a^i_k,\omega(a^i_k)},c_{a^i_k,\omega(a^i_k)})$.

\begin{proof} Given diagrams $x,y \in X^i$, the self-adjointness of $u_i$ implies that the value of $(x,y)$ is an invariant of the diagram $y \circ \omega(x)$. In particular,
$(x,y)=(\1,y\omega(x))=(x\omega(y),\1)$, where $y\omega(x)$ refers to the image of this diagram in the quotient $\mTL/I_i$. Therefore, if either $y \omega(x)$ or $x \omega(y)$ is not in
$X^i$ then the value of $(x,y)$ is zero. However, $X^i \cap \omega(X^i) = \{ c_{a^i_k,\omega(a^i_k)} \}$ where this set runs over all $k$ with $0 \le l \le l_i$. Thus the value of the
pairing on all elements is clearly determined by the values of $(\1, c_{a^i_k,\omega(a^i_k)})$ for all such $k$.

Consider the following map $V^i \times V^i \to \Zttt$: fix $k$, and for basis elements $x,y$ send $(x,y)$ to $r \in \Ztt$ if $y \omega(x) = r c_{a^i_k,\omega(a^i_k)} \in \mTL$, and send
$(x,y)$ to zero otherwise. Clearly this is a well-defined semi-linear map (being defined on a $\Ztt$-basis) and $u_j$ is self-adjoint. Thus we have enough pairings to prove independence.
\end{proof}

\begin{remark} Once again, all pairings are defined topologically. The closure of $c_{a^i_k,\omega(a^i_k)}$ has nesting number exactly $k$, which distinguishes the traces. \end{remark}

%
\subsection{Categorifying Cell Modules}
\label{subsec-catfy}
%

Categorifying the sign representation $L_{n+1}$ is easy. If we take the quotient of $\mTLC$ by all nonempty diagrams, we get a category where the only nonzero morphism space is the
one-dimensional space $\Hom(\emptyset,\emptyset)$. This clearly categorifies $L_{n+1}$, and we will say no more.

Consider the quotient of the category $\mTLC_1$ by all diagrams where any color not equal to $i$ appears on the left. Call this quotient $\mV^i_1$. As usual we let $\mV^i_2$ be its
additive grading closure, and $\mV^i$ its graded Karoubi envelope. We will show that $\mV^i \cong \mV^i_2$, so that we really may think of $\mV^i$ entirely diagrammatically without
worrying about idempotents. We claim that $\mV^i$ categorifies $V^i$. Not only this, but the action of $\mTLC$ on $\mV^i$ by placing diagrams on the right will categorify the action of
$\mTL$ on $V^i$.

Any monomial $u_{\ii}$ which goes to zero in $V^i$ is equal to a (scalar multiple of a) monomial $u_{\jj}$ where some index $j \ne i$ appears on the left. Therefore, the corresponding
object $U_{\ii}$ will be isomorphic to $U_{\jj}$, whose identity morphism is sent to zero in $\mV^i_1$ since it has a $j$-colored line on the left. There is an obvious map from $V^i$ to
the Grothendieck group of $\mV^i_2$, and the action of $\mTLC$, descended to the Grothendieck group, will commute with the action of $\mTL$ on $V^i$.

Therefore, $\Hom$ spaces in $\mV^i_1$ will induce a semi-linear pairing on $V^i$, which satisfies the property $(au_j,b)=(a,bu_j)$ because $U_j$ is self-adjoint. As before, once we
determine which pairing this is, our proof will be almost complete.

\begin{lemma} \label{pairinglemma} The pairing induced by $\mV^i_1$ will satisfy $(\1,c_{a^i_k,\omega(a^i_k)})=\frac{t^l}{1-t^2}$ where $k=n+1-2l$. \end{lemma}

\begin{remark} Taking a $(n+1,n+1)$ diagram and closing it off on the punctured plane, if $m$ is the number of circles and $k$ is the nesting number, then the pairing comes from the
trace on $\mTL$ which sends this configuration to $\qtwo^{l+m-(n+1)}\frac{t^l}{1-t^2}$.

For a closure of an arbitrary diagram, $l+m<n+1$ is possible. However, for any diagram in $X^i \cap \omega(X^i)$ (with extra circles thrown in) we have $l+m \ge n+1$, since removing the
circles yields precisely $c_{a^i_k,\omega(a^i_k)}$ for some $k$. This guarantees that evaluating the formula on an element of $V^i$ yields a power series with \emph{non-negative}
coefficients. \end{remark}

The proof of the lemma may be found shortly below. Temporarily assuming the lemma, the remainder of our results are easy.

\begin{thm} $\mV^i_2$ is idempotent closed and Krull-Schmidt, so that $\mV^i \cong \mV^i_2$. Its Grothendieck group is isomorphic to $V^i$. \end{thm}

\begin{proof} It is enough to check that for any $u_{\ii} \ne u_{\jj}$ corresponding to matchings in $X^i$, that $\Hom(U_{\ii},U_{\ii})$ is concentrated in non-negative degrees with a
1-dimensional degree 0 part, and that $\Hom(U_{\ii},U_{\jj})$ is concentrated in strictly positive degrees (see the proof of Proposition \ref{prop-traceisenough}). This is a calculation
using the semi-linear pairing.

Letting $m$ be the number of circles in a configuration on the punctured disk, and $k=n+1-2l$ the nesting number, then the evaluation will be in strictly positive degrees if $m<n+1$, and
will be in non-negative degrees with a 1-dimensional degree 0 part if $m=n+1$ exactly. But this was precisely the calculation in the proof of Lemma \ref{lemma-something}: for arbitrary
crossingless matchings $u_{\ii}$ and $u_{\jj}$, the closure of $u_{\ii} \omega(u_{\jj})$ has fewer than $n+1$ circles if $u_{\ii} \ne u_{\jj}$, and exactly $n+1$ if they're equal.
\end{proof}

\begin{cor} Let $\mL^i$ be the full subcategory of $\mV^i$ with objects consisting of (sums and grading shifts of) $U_{\ii}$ such that $u_{\ii}$ is an element of $V^i_{n+1-2l_i}$. This
has an action of $\mTLC$ on the right. On the Grothendieck group, this setup categorifies the cell module $L^i = V^i_{n+1-2l_i}$. \end{cor}

\begin{proof} That this subcategory is closed under the action of $\mTLC$ is obvious, as is the existence of a map from $L^i$ to the Grothendieck group. We already
know the induced pairing, because the subcategory is full. Therefore the same arguments imply that the Grothendieck group behaves as planned. \end{proof}

\begin{proof}[Proof of Lemma \ref{pairinglemma}] To calculate the pairing, we may calculate $(u_{\ii},u_{\ii}) = \gdim \End(U_{\ii})$ for the following choices of $\ii$:
$\emptyset$, $i$, $i(i+1)(i-1)$, $i(i+1)(i-1)(i+2)i(i-2)$, $i(i+1)(i-1)(i+2)i(i-2)(i+3)(i+1)(i-1)(i-3)$, etc. These are pictured below.

\igc{.8}{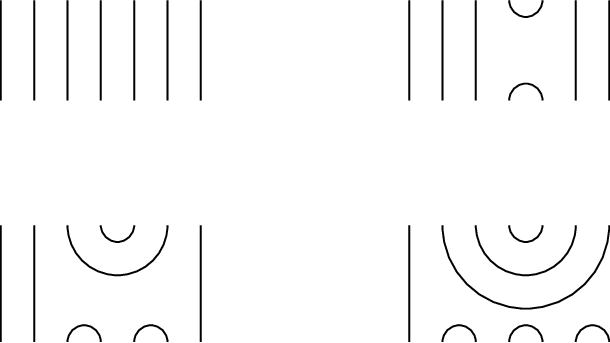}

These sequences are split into subsequences we call ``tiers", where the $m$th sequence adds the $m$th tier. The following property of these sequences is easily verified: each
sequence $\ii$ is in $X^i$, and remains in $X^i$ if one removes any subset of the final tier, but ceases to be in $X^i$ if one removes a single element from any other tier
instead.

Fix $\ii$ nonempty in this sequence, and let $\jj$ be the subsequence with the final tier removed. It is a quick exercise to show that the lemma is equivalent to $\gdim
\End(U_{\ii})=\frac{(1+t^2)^l}{1-t^2}$, where $l$ is the number of elements in the final tier.

Now consider an element of the endomorphism ring. Using previous results, we may assume it is a simple forest, with all double dots on the far left.
Any double dot colored $j \ne i$ will be sent to zero, so we have only an action of the ring $\Bbbk[f_i]$ on the left. This accounts for the $\frac{1}{1-t^2}$ appearing in all the
formulae.

Suppose there is a boundary dot in the morphism on any line not in the final tier. Then the morphism factors through the sequence $U_{\kk}$ where $\kk$ is $\ii$ with that index removed.
As discussed above, $u_{\kk}$ is not in $X^i$ and therefore $U_{\kk}$ is isomorphic to the zero object. See the first picture below for an intuitive reason why such a morphism vanishes.
Hence the only boundary dots which can appear occur on the final tier. It is easy to check that the existence of a trivalent vertex joining three boundary lines will force the existence
of a dot not on the final tier. See the second picture below. Both pictures are for the sequence $i(i+1)(i-1)(i+2)i(i-2)$, and blue will represent $i$ in all pictures in this section.

\igc{.8}{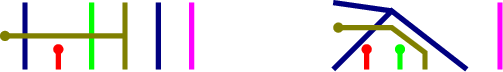}

Therefore a nonzero endomorphism must be $\1_{U_{\jj}}$ accompanied on the right by either identity maps or broken lines (pairs of boundary dots), because all other simple forests yield a
zero map. Identity maps have degree zero, while broken lines have degree 2. If these pictures form a basis (along with the action of blue double dots on the left), then the graded
dimension will be exactly as desired. An example with $l=3$, two broken lines, and one unbroken line is shown below.

\igc{.8}{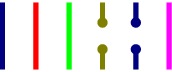}

This spanning set is linearly independent in $\mTLC$ over $\Bbbk$, so any further dependencies must come from having a non-blue color on the left. Consider an arbitrary endomorphism, and
reduce it using the $\mTLC$ relations to a simple forest with all double dots on the left. The actual double dots appearing are ambiguous, since there are polynomial relations in
$\mTLC$, but it is easy (knowing the generators of the TL ideal) to note that these relations are trivial modulo non-blue colors. Hence the spanning set will be linearly independent if
any diagram in $\mTLC$ which started with a non-blue color on the left will still have a non-blue color (perhaps in a double-dot) on the left after reducing to a simple forest. This will
be the case for any diagram with a boundary dot on $\jj$.

Let red indicate any other index, and suppose that red appears on the far left. Regardless of what index red is, unless there is a dot on $\jj$, the identity lines of $\jj$ block
this leftmost red component from reaching any red on the boundary of the graph. Take a neighborhood of a red line segment which includes no other colors and goes to $-\infty$.
Excising this neighborhood, we get a simply-connected region where the only relevant red boundary lines are the two which connect to the ends of the segment. Red then will reduce
to a simple forest with double dots on the left, which in this case yields either a red double dot or a red circle (potentially with more double dots).

\igc{1}{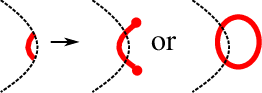}

However, no colors adjacent to red can interfere on the interior of a red circle, so the circle evaluates to zero. Therefore, the diagram evaluates to zero or has at least one red double
dot on the left. We may ignore the red double dot and reduce the remainder of the diagram, and so regardless of what else is done, the final result will have a red double dot on the
left. \end{proof}

\vspace{0.1in}
 
\noindent
{\sl \small Ben Elias, Department of Mathematics, Columbia University, New York, NY 10027}

\noindent 
{\tt \small email: belias@math.columbia.edu}

\end{document}